\documentclass[11pt,noindent]{article}

\usepackage{latexsym,amsmath,amsfonts,hyperref}

\newtheorem{theorem}{Theorem}
\numberwithin{theorem}{section}
\newtheorem{proposition}{Proposition}

\newtheorem{lemma}{Lemma}

\newtheorem{assumption}{Assumption}
\newtheorem{definition}{Definition}

\newtheorem{example}{Example}
\newtheorem{claim}{Claim}

\newtheorem{remark}{Remark}

\numberwithin{theorem}{section}
\numberwithin{lemma}{section}
\numberwithin{proposition}{section}

\newcommand{\thmref}[1]{Theorem~\ref{thm:#1}} 
\newcommand{\lemref}[1]{Lemma~\ref{lem:#1}} 
\newcommand{\propref}[1]{Proposition~\ref{prop:#1}} 
\newcommand{\claimref}[1]{Claim~\ref{claim:#1}} 
\newcommand{\secref}[1]{Section~\ref{sec:#1}} 
\newcommand{\eqnref}[1]{(\ref{eq:#1})} 

\def\be{\begin{equation} }
\def\ee{ \end{equation}}

\def\ben{\begin{equation*}}
\def\een{\end{equation*}}
\def\bea{\begin{eqnarray}}
\def\eea{\end{eqnarray}}
\def\ee{\end{eqnarray}}
\def\bean{\begin{eqnarray*}}
\def\eean{\end{eqnarray*}}

\newcommand\ignore[1]{}

\def\R{\mathbb{R}} 

\newcommand{\Ex}[1]{\mathbb{E}\left[#1\right]} 
\newcommand{\Prwo}[1]{\mathbb{P}} 
\newcommand{\Ind}[1]{\chi_{#1}} 

\renewcommand{\Pr}[1]{\mathbb{P}\left(#1\right)} 

\newcommand{\bigoh}[1]{O\left(#1\right)}
\newcommand{\liloh}[1]{o\left(#1\right)}

\def\sB{\mathcal{B}}\def\sC{\mathcal{C}}
\def\sF{\mathcal{F}}
\def\sG{\mathcal{G}}

\newcommand\QED{\ifhmode\allowbreak\else\nobreak\fi
\quad\nobreak$\Box$\medbreak}
\newcommand{\proofstart}{\par\noindent\sl Proof:\rm\enspace}
\newcommand{\proofend}{\QED\par}
\newenvironment{proof}{\proofstart}{\proofend}

\def\eps{\varepsilon}

\def\tr{{\rm Tr}}

\def\Sigmahat{\widehat{\Sigma}_n}

\def\bX{{\bf X}}
\def\hyp{{\sf h}}

\def\tr{{\rm tr}}

\def\ellhat{\widehat{\ell}_n}

\def\betahat{\widehat{\beta}_n}

\def\betamin{{\beta}_{\min}}

\def\ellhat{\widehat{\ell}_n}

\newcommand{\ip}[2]{ {#1}^T{#2}}

\def\Hi{\mathbb{R}^p}

\def\re{{\sf re}}

\def\bXhat{\widehat{\bf X}_n}

\begin{document}
\title{The lower tail of random quadratic forms, with applications to ordinary least squares and restricted eigenvalue properties} 
\author{Roberto Imbuzeiro Oliveira\thanks{IMPA, Rio de Janeiro, Brazil. \texttt{rimfo@impa.br}. Supported by a {\em Bolsa de Produtividade em Pesquisa} from CNPq, Brazil.  This article was produced as part of the activities of FAPESP Center for Neuromathematics (grant \# 2013/ 07699-0, S.Paulo Research Foundation).}} 
\maketitle

\begin{abstract}Finite sample properties of random covariance-type matrices  have been the subject of much research. In this paper we focus on the ``lower tail"'~of such a matrix, and prove that it is subgaussian under a simple fourth moment assumption on the one-dimensional marginals of the random vectors. A similar result holds for more general sums of random positive semidefinite matrices, and the (relatively simple) proof uses a variant of the so-called PAC-Bayesian method for bounding empirical processes.  

We give two applications of the main result. In the first one we obtain a new finite-sample bound for ordinary least squares estimator in linear regression with random design. Our result is model-free, requires fairly weak moment assumptions and is almost optimal. Our second application is to bounding restricted eigenvalue constants of certain random ensembles with ``heavy tails". These constants are important in the analysis of problems in Compressed Sensing and High Dimensional Statistics, where one recovers a sparse vector from a small umber of linear measurements. Our result implies that heavy tails still allow for the fast recovery rates found in efficient methods such as the LASSO and the Dantzig selector. Along the way we strengthen, with a fairly short argument, a recent result of Rudelson and Zhou on the restricted eigenvalue property. \end{abstract}

\section{Introduction}

Let $X_1,\dots,X_n$ be i.i.d. random (column) vectors in $\R^p$ with finite second moments. This paper contributes to the problem of obtaining finite-sample concentration bounds for the random covariance-type operator
\begin{equation}\label{eq:defsigmahat}\Sigmahat\equiv \frac{1}{n}\sum_{i=1}^nX_iX_i^T\end{equation} with mean $\Sigma\equiv \Ex{X_1X_1^T}$. This problem has received a great deal of attention recently, and has important applications to the estimation of covariance matrices \cite{SrivastavaVershynin2013,MendelsonPaouris2014}, to the analysis of methods for least squares problems \cite{HsuKakadeZhang2012} and to compressed sensing and high dimensional, small sample size statistics \cite{AdamczakRIP2011,RaskuttiEtAl2010,RudelsonZhou2013}.

The most basic problem is computing how many samples are needed to bring $\Sigmahat$ close to $\Sigma$. One needs at least $n\geq p$ to bring $\Sigmahat$ close to $\Sigma$, so that the ranks of the two matrices can match. A basic problem is to find conditons under which $n\geq C(\eps)\,p$ samples are enough for guaranteeing 
\begin{equation}\label{eq:goalintro}\Pr{\forall v\in\R^p,\, (1-\eps)v^T\Sigma v \leq v^T\Sigmahat\,v\leq (1+\eps)\,v^T\Sigma\,v}\approx 1,\end{equation} where $C(\eps)$ depends only on $\eps>0$ and on moment assumptions on the $X_i$'s. 

A well known bound by Rudelson \cite{Rudelson1999,Oliveira2010} implies $C(\eps)\,p\log p$ samples are necessary and sufficient if the vectors $\Sigma^{-1/2}X_i/\sqrt{p}$ have uniformly bounded norms. Removing the $\log p$ factor is relatively easy for subgaussian vectors $X_i$, but even the seemingly nice case of logconcave random vectors (which have subexponential moments) had to wait for the breakthrough papers by Adamczak et al \cite{AdamczakEtAl2010,AdamczakEtAl2011}. The current best results hold when the $X_i$ and all of their projections have $q>2$ moments  \cite{SrivastavaVershynin2013}, and when their one-dimensional marginals have $q>8$ moments \cite{MendelsonPaouris2014}; in the latter case one also needs (necessarily) a high probability bound on $\max_{i\leq n}\,|X_i|$. None of those finite-moment results gives strong concentration bounds. 

It turns out that, for many important applications, only the {\em lower tail} of $\Sigmahat$ matters. That is, we only need that $v^T \Sigmahat v$ is not much smaller than $v^T \Sigma v$ for all vectors $v$ in a suitable set. Our main result in this paper is that this lower tail is subgaussian under extremely weak conditions. More precisely, we will prove that if there exists a $\hyp>0$ such that
\begin{equation}\label{eq:mainassumption}\forall v\in\R^p\,:\, \sqrt{\Ex{(v^TX_1)^4}}\leq \hyp\,v^T\Sigma\, v,\end{equation} then
 $n=\bigoh{\hyp^2\,p/\eps^2}$ samples are enough to guarantee an asymmetric version of \eqnref{goalintro}, to wit:
\begin{equation}\label{eq:asymmetric}\Pr{\forall v\in\R^p\,:\, v^T\Sigmahat\,v\geq (1-\eps)\,v^T\Sigma\,v}\geq 1-e^{-p}.\end{equation}
This follows from a more precise result -- \thmref{main} in \secref{main} below -- about the more general case of sums of independent and identically distributed positive semidefinite random matrices. We note that the dependence on $\eps^{-2}$ in our bound is optimal for vectors with independent coordinates, as can be shown via the Bai-Yin theorem \cite{BaiYin1993}. 

We will give two applications to illustrate our main result. One is to {\em least squares linear regression with random design}, which we discuss in \secref{OLS}. In this problem one is given data in the form of $n$ i.i.d. copies of a random pair $(X,Y)\in\R^p\times \R$, and the goal is to find some $\betahat\in\R^p$ such that $X^T\betahat$ is as good a $L^2$ approximation to $Y$ as possible. The most basic method for this problem is the ordinary least squares estimator, and recent finite-sample bounds by Hsu et al. \cite{HsuKakadeZhang2012} and Audibert and Catoni \cite{AudibertCatoni2011} have shown that the error of the ordinary least squares method is $\bigoh{\sigma^2\,p/n}$, where $\sigma^2$ measures the intensity of the noise. Both results hold in a model-free setting, where the data generating mechanism is not assumed to correspond to a linear model, but their assumptions are stringent in that they involve infinitely many moments $X$ and/or $Y$. We prove here a result -- \thmref{OLS} below -- that gives improved bounds under weaker assumptions. In particular, it seems to be the first bound of this form that only assumes finitely many moments of $X$ and $Y$.  

The second application, discussed in \secref{re}, deals with so-called {\em restricted eigenvalue constants}. These values quantify how $\Sigmahat$ acts on vectors $v\in\R^p$ which are constrained to have a positive fraction of their $\ell_1$ norm on a set of $s\ll p$ coordinates. Restricted eigenvalues are used in the analysis of Compressed Sensing and High Dimensional Statistics problems, where one wants to estimate a vector $\betamin\in\R^p$ from a number $n\ll p$ of linear measurements $X_i^T\betamin$. Estimators such as the LASSO and the Dantzig selector \cite{Tibshirani1996,CandesTao2007} have been analyzed under the condition that $\betamin$ is sparse (with $s\ll n/\log p$ nonzero coordinates) and the linear measurement vectors have positive restricted eigenvalues \cite{BickelRT2009,BuhlmannVanDerGeer2008}. It is thus natural to enquire whether random ensembles satisfy this property \cite{RaskuttiEtAl2010,RudelsonZhou2013}. \thmref{re} shows that this property may be expected even when the measurement vectors have relatively heavy tails, as long as the sparsity parameter $s$ satisfies 
$s\log p=\liloh{n}$ and one ``normalizes"~the matrix $\Sigmahat$ (which is in fact quite natural). In particular, we sketch in \secref{linearsparse} what this result implies for random design linear regression when $p\gg n$.

Let us briefly comment on some proof ideas we think might be useful elsewhere. \thmref{main}, our main result, is proven via so called PAC Bayesian methods and is inspired by the recent paper by Audibert and Catoni \cite{AudibertCatoni2011}. We will see that this method allows one to translate properties of moment generating functions of individual random variables into uniform control of certain empirical processes. This is discussed in more detail in \secref{proofideas}. 

Later on, when we move to the problem of restricted eigenvalues, we will see that we need to control $v^T\Sigmahat\,v$ uniformly over vectors satisfying certain $\ell_1$ norm constraits. We will prove a ``transfer principle" (\lemref{transfer} below ) that implies that this control can be deduced from a (logically) weaker control of $\Sigmahat$ over sparse vectors. In spite of its very short proof, this result is stronger than a similar theorem in a recent paper by Rudelson and Zhou \cite{RudelsonZhou2013}; this connection is discussed in Appendix \ref{sec:appendix.rudelson}.\\

\noindent{\sc Organization:} The next section covers some preliminaries and defines the notation we use. \secref{main} contains the statement and proof of the main result, \thmref{main}, along with a discussion of the assumptions and a proof overview. \secref{OLS} presents our result on ordinary least squares, giving some background for the problem.  \secref{re} follows a similar format for restricted eigenvalues. The final section presents some remarks and open problems. Two Appendices contain a discussion of our improvement over \cite{RudelsonZhou2013}, and some estimates used in the main text.

\section{Notation and preliminaries}\label{sec:prelim}
The coordinates of a vector $v\in\R^p$ are denoted by $v[1],v[2],\dots,v[p]$. The support of $v$ is the set:
$${\sf supp}(v)\equiv \{1\leq j\leq p\,:\, v[j]\neq 0\}.$$
The restriction of $v$ to a subset $S\subset \{1,\dots,p\}$ is the vector $v_S$ with $v_S[j]=v[j]$ for $j\in S$ and $v_S[k]=0$ for $k\not\in S$.

The $\ell_0$ norm of $v$, denoted by $|v|_0$, is simply the cardinality of ${\sf supp}(v)$. For $q\geq 1$, the $\ell_q$ norm is defined as:
$$|v|_q\equiv \sqrt[q]{\sum_{j=1}^p|v[j]|^q}.$$

$\R^{p\times p'}$ is the space of matrices with $p$ rows, $p'$ columns and real entries. Given $A\in\R^{p\times p'}$, we denote by $A^T$ its transpose. $A$ is symmetric if $A=A^T$. Given $A\in\R^{p\times p}$ we let $\tr(A)$ denote the trace of $A$ and $\lambda_{\max}(A)$ denote its largest eigenvalue. The $p\times p$ identity matrix is denoted by $I_{p\times p}$.
We identify $\R^p$ with the space of column vectors $\R^{p\times 1}$, so that the standard Euclidean inner product of $v,w\in\R^p$ is $v^Tw$. 

We say that $A\in\R^{p\times p}$ is positive semidefinite if it is symmetric and $v^TAv\geq 0$ for all $v\in\R^p$. In this case one can easily show that
\begin{equation}\label{eq:psd}v^TAv=0\Leftrightarrow v^TA=0\Leftrightarrow Av=0.\end{equation}
The $2\to 2$ norm of $A\in \R^{p\times p'}$ is 
$$|A|_{2\to 2}\equiv \max_{v\in\R^{p'}\,:\, |v|_2=1}|Av|_2.$$
For symmetric $A\in\R^{p\times p}$ this is the largest absolute value of its eigenvalues. Moreover, if $A$ is positive semidefinite $|A|_{2\to 2}=\lambda_{\max}(A)$. If $A$ is symmetric and invertible, we also have
\begin{equation}\label{eq:norminverse}|A^{-1}|_{2\to 2} = \frac{1}{\min_{v\in\R^p\,:\, |v|_2=1}v^TAv}.\end{equation}

We use asymptotic notation somewhat informally, in order to illustrate our results with clean statements. We write $a=\liloh{b}$ or $a\ll b$ to indicate that $|a/b|$ is very small, and $a=\bigoh{b}$ to say that $|a/b|$ is bounded by a universal constant.  

Finally, we state for later use the Burkholder-Davis-Gundy inequality. Let $(M_i,\sF_i)_{i=1}^n$ denote a martingale with finite $q$-th moments ($q\geq 2$) and $M_0=0$ . Then:
\begin{equation}\label{eq:BDG}\Ex{|M_n|^q}^{\frac{1}{q}}\leq q\,\Ex{\left(\sum_{i=1}^n(M_i-M_{i-1})^2\right)^{\frac{q}{2}}}^{\frac{1}{q}}\leq q\,\sqrt{n}\,\max_{1\leq i\leq n}\Ex{|M_i-M_{i-1}|^q}^{\frac{1}{q}}.\end{equation}
Note that the first inequality above is the BDG inequality with optimal constant, and the second inequality follows from Minkowski's inequality for the $L^{q/2}$ norm. We also observe that \eqnref{BDG} implies a result for $W_1,\dots,W_n$ which are i.i.d. random variables:
\begin{equation}\label{eq:BDGiid}\Ex{\left|\frac{1}{n}\sum_{i=1}^nW_i - \Ex{W_1}\right|^q}^{\frac{1}{q}}\leq \frac{q}{\sqrt{n}}\,\Ex{|W_1-\Ex{W_1}|^q}^{1/q}\leq \frac{2q}{\sqrt{n}}\,\Ex{|W_1|^q}^{\frac{1}{q}}.\end{equation}
Better inequalities are known in this case, but we will use \eqnref{BDGiid} for simplicity.

\section{The subgaussian lower tail}\label{sec:main}

The goal of this section is to discuss and prove our main result. 

\begin{theorem}[Proven in \secref{mainproof}]\label{thm:main}Assume $A_1,\dots,A_n\in\R^{p\times p}$ are i.i.d. random positive semidefinite matrices whose coordinates have bounded second moments. Define $\Sigma\equiv \Ex{A_1}$ (this is an entrywise expectation) and 
$$\Sigmahat\equiv \frac{1}{n}\sum_{i=1}^nA_i.$$ Let  $\hyp\in (1,+\infty)$ be such that $\sqrt{\Ex{(v^TA_1\,v)^2}}\leq \hyp\,v^T\,\Sigma v$ for all $v\in \R^p$. Then for any $\delta\in (0,1)$:
$$\Pr{\forall v\in\R^p\,:\, v^T\Sigmahat\,v\geq \left(1 - 7\hyp\,\sqrt{\frac{p+2\ln(2/\delta)}{n}}\right)\,v^T\Sigma v}\geq 1-\delta.$$\end{theorem} 

Notice that a {\em particular case} of this Theorem is when $A_i=X_iX_i^T$ where $X_1,\dots,X_n\in\R^p$ are i.i.d.. Therefore \thmref{main} corresponds to our discussion in the Introduction. In what follows we discuss what our assumption \eqnref{mainassumption} entails and when it is verified. We then discuss the main ideas in the proof, and finally move to the proof itself.

\subsection{On the assumption}\label{sec:assumption}

Let us recall that in the vector case $A_i=X_iX_i^T$ the main assumption we need is that 
\begin{equation}\label{eq:mainassumption2}\forall v\in\R^p\,:\, \sqrt{\Ex{(v^TX)^4}}\leq \hyp\,v^T\Sigma\, v,\end{equation}
for some $1<\hyp<+\infty$, where $X\equiv X_1$ and $\Sigma\equiv \Ex{XX^T}$. Note that an inequality in the opposite direction always holds, thanks to Jensen's inequality:
$$\Ex{(v^TX)^2}\leq \sqrt{\Ex{(v^TX)^4}}.$$
Assumption \eqnref{mainassumption2} is invariant by linear transformations: that is, if $A\in \R^{p'\times p}$ then  $\tilde{X}\equiv AX$ also satisfies \eqnref{mainassumption2}. Moreover, if $\Ex{X}=0$ then translating $X$ by some $b\in\R^p$ will only at most increase $\hyp$ to $2\sqrt{2}\,\hyp$. 

An obvious case where \eqnref{mainassumption2} holds is when $X[1],\dots,X[p]$ are independent, have finite fourth moments and mean $0$. A short calculation shows that we may take 
$$\hyp\equiv 6\vee \left(\max_{1\leq j\leq p\,:\Ex{X[j]^2}>0}\frac{\sqrt{\Ex{X[j]^4}}}{\Ex{X[j]^2}}\right).$$
Significantly, the same calculations also work when $X[1],\dots,X[p]$ are four-wise independent; this will be interesting when considering compressed sensing-type applications (cf. Example \ref{ex:efficient} below). Changing to $2\sqrt{2}\,\hyp$ allows us to consider translations and linear transformations of $X$. 

These particular cases include many important examples, such as gaussian, subgaussian, logconcave vectors and their affine transformations. There are also many examples with unbounded $4+\eps$ moments. If we multiply $X$ by an independent scalar 
$\xi$ with 
$$\Ex{\xi^4}\leq \hyp_*^2\Ex{\xi^2};$$
we just need to replace $\hyp$ with $\hyp\,\hyp_*$. Interestingly, the {\em upper tail} of $\Sigmahat$ is quite sensitive to this kind of transformation. Even multiplying by a Gaussian random variable may result in an ensemble that does {\em not} obey the analogue of the main theorem (cf. the discussion in \cite[Section 1.8]{SrivastavaVershynin2013}).

\subsection{Proof overview and a preliminary PAC Bayesian result}\label{sec:proofideas}

At first sight it may seem odd that we can obtain such strong concentration from finite moment assumptions. The key point here is that, for any $v\in\R^p$, the expression $$v^T\Sigmahat v = \frac{1}{n}\sum_{i=1}^n\,(X_i^Tv) ^2$$ is a sum of random variables which are independent, identically distributed and {\em non negative}. Such sums are well known to have subgaussian lower tails under weak assumptions; see eg. \lemref{concnonneg} below.  

This fact may be used to show concentration of $v^T\Sigmahat\,v$ for any fixed $v\in\R^p$. It is less obvious how to turn this into a uniform bound. The standard techniques for this, such as chaining, involve looking at a discretized subset of $\R^p$ and moving from this finite set to the whole space. In our case this second step is problematic, because it requires {\em upper} bounds on $v^T\Sigmahat\,v$, and we know that our assumptions are not strong enough to obtain this. 

What we use instead is the so-called PAC Bayesian method \cite{CatoniBook} for controlling empirical processes. At a very high level this method replaces chaining and union bounds with arguments based on the relative entropy. What this means in our case is that a ``smoothened-out"~version of the process $v^T\Sigmahat\,v$ ($v\in\R^p$), where $v$ is averaged over a Gaussian measure, automatically enjoys very strong concentration properties. This implies that the original process is also well behaved as long as the effect of the smoothing can be shown to be negligible. Many of our ideas come from Audibert and Catoni \cite{AudibertCatoni2011}, who in turn credit Langford and Shawe-Taylor \cite{LangfordShaweT2002} for the idea of Gaussian smoothing. 

To make these ideas more definite we present a technical result that encapsulates the main ideas in our PAC Bayesian approach. This requires some conditions. 

\begin{assumption}\label{assum:Ztheta}$\{Z_\theta\,:\,\theta\in\R^p\}$ is a family of random variables defined on a common probability space $(\Omega,\sF,\mathbb{P})$. We assume that the map $$\theta\mapsto Z_\theta(\omega)\in\R$$ is continuous for each $\omega\in\Omega$. Given $v\in\R^p$ and a positive semidefinite $C\in\R^{p\times p}$, we let $\Gamma_{v,C}$ denote the Gaussian probability measure over $\R^p$ with mean $v$ and covariance matrix $C$. We will also assume that for all $\omega\in\Omega$ the integrals
$$(\Gamma_{v,C}\,Z_\theta)\,(\omega)\equiv \int_{\R^p}Z_\theta(\omega)\,\Gamma_{v,C}(d\theta)$$
are well defined and depend continuously on $v$. We will use the notation $\Gamma_{v,C}f_\theta$ to denote the integral of $f_\theta$ (which may also depend on other parameteres) over the variable $\theta$ with the measure $\Gamma_{v,C}$.\end{assumption}

\begin{proposition}[PAC Bayesian Proposition]\label{prop:pacbayesian}Assume the above setup, and also that $C$ is invertible and $\Ex{e^{Z_\theta}}\leq 1$ for all $\theta\in \R^d$. Then for any $t\geq 0$,
$$\Pr{\forall v\in \R^p\,:\, \Gamma_{v,C}Z_\theta\leq t + \frac{|C^{-1/2}v|_2^2}{2}}\geq 1 - e^{-t}.$$\end{proposition}

In the next subsection we will apply this to prove \thmref{main}. Here is a brief overview: we will performe a change of cordinates under which $\Sigma=I_{p\times p}$. We will then define $Z_\theta$ as
$$Z_\theta = \xi  |\theta|^2_2-\xi \theta^T\Sigmahat\,\theta + \mbox{(other terms)}$$
where $\xi>0$ will be chosen in terms of $t$ and the ``other terms"~ will ensure that $\Ex{e^{Z_\theta}}\leq 1$. Taking $C=\gamma\,I_{p\times p}$ will result in
$$\Gamma_{v,C}Z_\theta = \xi  |v|^2_2-\xi v^T\Sigmahat\,v+ S_v +\mbox{(other terms)}$$
where
$$S_v\equiv \gamma\,p -\gamma\,\tr(\Sigmahat)$$
is a new term introduced by the``smoothing operator" $\Gamma_{v,\gamma C}$. The choice $\gamma=1/p$ will ensure that this term is small, and the ``other terms"~will also turn out to be manageable. The actual proof will be slightly complicated by the fact that we need to truncate the operator $\Sigmahat$ to ensure that $S_v$ is highly concentrated.

\begin{proof}[of \propref{pacbayesian}] As a preliminary step, we note that under our assumptions the map:
$$\omega\in \Omega\mapsto \sup_{v\in \R^d}\Gamma_{v,C}Z_\theta(\omega) - \frac{|C^{-1/2}v|_2^2}{2}\in \R\cup \{+\infty\}$$
is measurable. This implies that the event in the statement of the proposition is indeed a measurable set.

To continue, recall the definition of Kullback Leiber divergence (or relative entropy) for probability measures over a measurable space $(\Theta,\sG)$:
\begin{equation}\label{def:KL}K(\mu_{1}|\mu_0)\equiv \left\{\begin{array}{ll}\int_\Theta\,\ln\left(\frac{d\mu_{1}}{d\mu_0}\right) \,d\mu_1, & \mbox{if $\mu_1\ll \mu_0$}; \\+\infty, & \mbox{otherwise}.\end{array}\right.\end{equation}
A well-known variational principle \cite[eqn. (5.13)]{LedouxConcentration} implies that for any measurable function $h:\Theta\to\R$:
\begin{equation}\label{eq:variational}\int h\,d\mu_1\leq \ln\left(\int e^{h}\,d\mu_0\right) + K(\mu_{1}|\mu_0).\end{equation}
We apply this when $(\Theta,\sG) =(\R^d,\sB(\R^d))$, $\mu_1=\Gamma_{v,C}$, $\mu_0=\Gamma_{0,C}$ and $h=Z_\theta$. In this case it is well-known that the relative entropy of the two measures is $|C^{-1/2}v|_2^2/2$. This implies: 
$$\sup_{v\in \R^p} \left(\Gamma_{v,C}Z_\theta-\frac{|C^{-1/2}v|_2^2}{2}\right)\leq  \ln\left(\Gamma_{0,C}\,e^{Z_\theta}\right).$$
To finish, we prove that:
$$\Pr{\Gamma_{0,C}\,e^{Z_\theta}\geq e^{t}}\leq e^{-t}.$$
But this follows from Markov's inequality and Fubini's Theorem:
$$\Pr{\Gamma_{0,C}\,e^{Z_\theta}\geq e^{t}}\leq e^{-t}\,\Ex{\Gamma_{0,C}\,e^{Z_\theta}} = e^{-t}\,\Gamma_{0,C}\Ex{e^{Z_\theta}}\leq e^{-t},$$
because $\Ex{e^{Z_\theta}}\leq 1$ for any fixed $\theta$.\end{proof}

\subsection{Proof of the main result}\label{sec:mainproof}

\begin{proof}[of \thmref{main}] We will assume throughout the proof that $\Sigma$ is invertible. If that is not the case,  we can restrict ourselves to the range of $\Sigma$, which is isometric to $\R^{p'}$ for some $p'\leq p$, noting that $A_iv=0$ and $v^TA_i=0$ almost surely for any $v$ that is orthogonal to the range (this follows from $\Ex{v^TA_1v}=0$ for $v$ orthogonal to the range, combined with \eqnref{psd} above). 

Granted invertibility, we may define:
\begin{equation}\label{eq:defBi}B_i\equiv \Sigma^{-1/2}A_i\Sigma^{-1/2}\,\,(1\leq i\leq n)\end{equation}
and note that $B_1,\dots,B_n$ are i.i.d. positive semidefinite with $\Ex{B_1}=I_{p\times p}$. Moreover, 
\begin{equation}\label{eq:hypB}\forall v\in \R^p\,:\, \sqrt{\Ex{(v^TB_1v)^2}} =   \sqrt{\Ex{((\Sigma^{-1/2}v)^TA_1\,(\Sigma^{-1/2}v))^2}}\leq \hyp\,|v|^2_2.\end{equation} 

The goal of our proof is to show that, for any $t\geq 0$:
\begin{equation}\label{eq:goalAi}{\bf Goal: }\Pr{\forall v\in \R^p\,:\, \sum_{i=1}^n\frac{\ip{v}{A_i v}}{n}\geq \ip{v}{\Sigma\,v} - 7\hyp\,\sqrt{\frac{p+2t}{n}}\,\ip{v}{\Sigma v}}\geq 1-2e^{-t},\end{equation}
Replacing $v$ with $\Sigma^{-1/2}v$ above and using homogeneity reduces this goal to showing:
\begin{equation}\label{eq:goalBi1}{\bf Goal: }\forall t\geq 0,\, \Pr{\forall v\in \R^p\,:\, |v|_2=1\Rightarrow \sum_{i=1}^n\frac{\ip{v}{B_i v}}{n}\geq 1 - 7\hyp\sqrt{\frac{p+2t}{n}}}\geq 1-2e^{-t}.\end{equation}
This is what we will show in the remainder of the proof.

Fix some $R>0$ and define (with hindsight) truncated operators 
\begin{equation}\label{eq:Bitruncated}B_i^R\equiv \left(1\wedge \frac{R}{\tr(B_i)}\right)\,B_i,\end{equation}
with the convention that this is simply $0$ if $\tr(B_i)=0$. We collect some estimates for later use.

\begin{lemma}[Proven subsequently]\label{lem:esttruncate}We have for all $v\in\R^p$ with $|v|_2=1$
$$\frac{1}{n}\sum_{i=1}^n\ip{v}{B_i^Rv}\leq \sum_{i=1}^n\ip{v}{B_iv};$$
$$\Ex{\ip{v}{B_i^Rv}}\geq 1 - \frac{\hyp^2\,p}{R}.$$
Moreover,
$$\Ex{\tr(B_i^R)^2}\leq \hyp^2\,p^2.$$\end{lemma}

Fix $\xi>0$. We will apply \propref{pacbayesian} with $C=I_{p\times p}/p$ and
$$Z_\theta\equiv \xi\,\Ex{\ip{\theta}{B_1^R\theta}} - \xi \sum_{i=1}^n\frac{\ip{\theta}{B_i^R\theta}}{n} - \frac{\xi^2}{2n}\,\Ex{(\ip{\theta}{B_i^R\theta})^2}.$$
The continuity and integrability assumptions of the Proposition are trivial to check. The assumption $\Ex{e^{Z_\theta}}\leq 1$ follows from independence, which implies:
$$\Ex{e^{Z_\theta}} = \prod_{i=1}^n\Ex{e^{\frac{\xi\,\,\Ex{\ip{\theta}{B_1^R\theta}}}{n} -  \frac{\xi\ip{\theta}{B^R_i\theta}}{n} - \frac{\xi^2}{2n^2}\,\Ex{(\ip{\theta}{B_i^R\theta})^2}}}$$
plus the fact that, for any non-negative, square-integrable random variable $W$,
$$\Ex{e^{\xi\,\Ex{W}-\xi\,W - \frac{\xi^2}{2}\,\Ex{W^2}}} \leq 1$$
(this is shown in the proof of \lemref{concnonneg} in the Appendix). 
We deduce from \propref{pacbayesian} that, with probability $\geq 1-e^{-t}$,
$$\forall v\in\R^d\,:\, \xi\,\Gamma_{v,C}\Ex{\ip{\theta}{B_1^R\theta}}- \xi \sum_{i=1}^n\Gamma_{v,C}\frac{\ip{\theta}{B_i^R\theta}}{n} - \frac{\xi^2}{2n}\,\Ex{(\ip{\theta}{B_i^R\theta})^2}\leq \frac{p|v|_2^2 + 2t}{2},$$
which is the same as saying that, with probability $\geq 1-e^{-t}$, the following inequality holds for all $v\in\R^p$ with $|v|_2=1$:
\begin{equation}\label{eq:relevantevent} \sum_{i=1}^n\Gamma_{v,C}\frac{\ip{\theta}{B_i^R\theta}}{n}\geq \Gamma_{v,C}\Ex{\ip{\theta}{B_1^R\theta}} - \left(\frac{\xi}{2n}\,\Gamma_{v,C}\Ex{(\ip{\theta}{B_i^R\theta})^2}+\frac{p + 2t}{2\xi}\right).\end{equation}
Let us now compute all the integrals with respect to $\Gamma_{v,C}$ that appear above, for $v\in \R^p$ with $|v|_2=1$:
\begin{eqnarray}\nonumber \frac{1}{n}\sum_{i=1}^n\Gamma_{v,C}\,\ip{\theta}{B_i^R\theta}&=& \frac{1}{n}\,\sum_{i=1}^n\ip{v}{B_i^Rv} + \sum_{i=1}^n\frac{\tr(B_i^R)}{p n}\\ \label{eq:sometermnonneg}\mbox{(use \lemref{esttruncate})}&\leq & \frac{1}{n}\,\sum_{i=1}^n\ip{v}{B_iv} + \sum_{i=1}^n\frac{\tr(B_i^R)}{p n};\\ \nonumber
\Gamma_{v,C}\,\Ex{\ip{\theta}{B_1^R\theta}}&=& \Ex{\ip{v}{B_1^Rv}}+ \frac{\Ex{\tr(B_1^R)}}{p }\\ \label{eq:easypartnonneg}\mbox{(use \lemref{esttruncate})}&\geq & 1- \frac{\hyp^2\,p}{R} + \frac{\Ex{\tr(B_i^R)}}{p}.\end{eqnarray}

We also need estimates for $\Gamma_{v,C}\Ex{(\ip{\theta}{B_i^R\theta})^2}$. Standard calculations with the normal distribution show that:
$$\Gamma_{v,C}(\ip{\theta}{B_i^R\theta}) ^2 = \Gamma_{0,C}\,(\ip{v}{B_i^Rv} + \ip{\theta}{B_i^R\theta} + 2\ip{\theta}{B_i^Rv})^2.$$
The first two terms inside the brackets are non-negative and, by Cauchy Schwartz, the absolute value of the rightmost term is at most the sum of the other two. We deduce:
\begin{eqnarray*}\Gamma_{v,C}\ip{\theta}{B_i^R\theta}^2&\leq& \Gamma_{0,C}\,(2\ip{v}{B_i^Rv} + 2\ip{\theta}{B_i^R\theta})^2 \\ & =& 4(\ip{v}{B_i^Rv})^2 + 4 \Gamma_{0,C}(\ip{\theta}{B_i^R\theta})^2 + 8\ip{v}{B_i^Rv}\,\Gamma_{0,C}\ip{\theta}{B_i^R\theta}\\ &\leq & 4(\ip{v}{B_i^Rv})^2 + 12\frac{\tr(B_i^R)^2}{p^2} + 8\ip{v}{B_i^Rv}\,\frac{\tr(B_i^R)}{p}.\end{eqnarray*}  Taking expectations, applying \lemref{esttruncate} and recalling $|v|_2=1$ gives:
\begin{eqnarray}\nonumber \Gamma_{v,C}\Ex{\ip{\theta}{B_i^R\theta}^2}&\leq & 16\hyp^2 + \frac{8}{p}\,\Ex{\ip{v}{B_i^Rv}\,\tr(B_i^R)}\\ \nonumber&\leq & 16\hyp^2 + \frac{8}{p}\,\sqrt{\Ex{(\ip{v}{B_i^Rv})^2}\Ex{\tr(B_i^R)^2}} \\ \label{eq:squaremain}&\leq &24\hyp^2.\end{eqnarray}

We plug this last estimate into \eqnref{relevantevent} together with \eqnref{easypartnonneg} and  \eqnref{sometermnonneg}. This results in the following inequality, which holds with probability $\geq 1-e^{-t}$ simultaneously for all $v\in \R^d$ with $|v|_2=1$:
$$\frac{1}{n}\sum_{i=1}^n\ip{v}{B_i^Rv}\geq 1 -\frac{\hyp^2p}{R} + \left(\sum_{i=1}^n\frac{\tr(B_i^R)-\Ex{\tr(B_i^R)}}{p n}\right) + \frac{24\xi\hyp^2}{2n} + \frac{p+2t}{2\xi}.$$
This holds for any choice of $\xi$. Optimizing over this parameter shows that, with probability $\geq 1-e^{-t}$, we have the following inequality simultaneously for all $v\in\R^p$ with $|v|_2=1$.
\begin{equation}\label{eq:probboundbeforeR} \sum_{i=1}^n\frac{\ip{v}{B_iv}}{n}\geq  1- \frac{\hyp^2p}{R} + \left(\sum_{i=1}^n\frac{\tr(B_i^R)-\Ex{\tr(B_i^R)}}{p n}\right) + \hyp\,\sqrt{\frac{24\,(p+2t)}{n}}.\end{equation}
We now take care of the term between curly brackets in the RHS. This is precisely the moment when the truncation of $B_i$ is useful, as it allows for the use of Bennett's concentration inequality. More specifically, note that the term under consideration is a sum of iid random variables that lie between $-R/p n$ and $R/p n$. Moreover, the variance of each term is at most $\Ex{\tr(B_i^R)^2}/p^2 n^2\leq \hyp^2/n^2$ by \lemref{esttruncate}. We may use Bennett's inequality to deduce that with probability $\geq 1-e^{-t}$:
$$\sum_{i=1}^n\frac{\tr(B_i^R)-\Ex{\tr(B_i^R)}}{p n}\leq \hyp\sqrt{\frac{2\,t}{n}} + \frac{2R\,t}{3p n}.$$
Combining this with \eqnref{probboundbeforeR} implies that, for any $t\geq 0$, the following inequality holds with probability $\geq 1-2e^{-t}$, simultaneously for all $v\in \R^p$ with $|v|_2=1$:
$$1- \sum_{i=1}^n\frac{\ip{v}{B_iv}}{n}\leq \frac{\hyp^2 p}{R} + \frac{2R\,t}{3p n} + \hyp\,\sqrt{\frac{2\,t}{n}} + \hyp\,\sqrt{\frac{24(p+2t)}{n}}$$
This holds for any $R>0$. Optimizing over $R$ gives:
$$\Pr{\begin{array}{l}\forall v\in\Hi\\ \mbox{ w/ }|v|_2=1\end{array}\,:\,1- \sum_{i=1}^n\frac{\ip{v}{B_iv}}{n}\leq 2\hyp\,\sqrt{\frac{2t}{3n}} +\hyp\,\sqrt{\frac{24(p+2t)}{n}}}\geq 1-2e^{-t}.$$
The overestimates $2/3\leq 1$, $24\leq 5^2$ and $0\leq p$ finish the proof of \eqnref{goalBi1}. This in turn finishes the proof of \thmref{main} except for \lemref{esttruncate}, which is provn below.\end{proof}

\begin{proof}[of \lemref{esttruncate}] The first item is immediate. The third item follows from $\tr(B_i^R)\leq \tr(B_i)$ and \lemref{powertrace} in \secref{powertrace}. 

To finish, we prove the second assertion. Fix some $v\in \Hi$ with norm one. We have
$$\Ex{\ip{v}{B_i^Rv}}\geq \Ex{\ip{v}{B_iv} (1 - \Ind{\{\tr(B_i)>R\}})}\geq 1  - \sqrt{\Ex{(\ip{v}{B_iv})^2}\,\Pr{\tr(B_i)>R}}$$
by Cauchy Schwartz. Now note that 
$$\Ex{(\ip{v}{B_iv})^2}\leq\hyp^2|v|_2^2=\hyp^2.$$
Moreover, by the previous estimate on $\Ex{\tr(B_i)^2}$,
$$\Pr{\tr(B_i)>R}\leq\frac{\Ex{\tr(B_i)^2}}{R^2} \leq \frac{\hyp^2\,p^2}{R^2}.$$
Combining the last three inequalities finishes the proof.\end{proof}

\begin{remark}It is instructive to compare this proof with what one would obtain without truncation. In that case everything would go through {\em except} for the step where we apply Bennett's inequality. \end{remark}

\section{Ordinary least squares under random design}\label{sec:OLS}

\subsection{Setup}\label{sec:OLSsetup} 
{\em Linear regression with random design} is a central problem in Statistics and Machine Learning. In it one is given data in the form of $n$ independent and identically distributed copies $\{(X_i,Y_i)\}_{i=1}^n$ of a square-integrable pair $(X,Y)\in\R^p\times \R$, where $X$ is a vector of so-called {\em covariates} and $Y$ is a {\em response variable}. The goal is to find a vector $\betahat\in\R^p$, which is a function of the data, which makes the square loss  
$$\ell(\beta)\equiv \Ex{(Y-X^T\beta)^2}\,\,\,(\beta\in\R^p)$$
as small as possible. In other words, one is trying to find a linear combination of the coordinates of $X$ that is as close as possible to $Y$ in terms of mean-square error. The random design setting should be contrasted with the technically simpler case of fixed design, where the $X_i$'s are assumed fixed and all randomness is in the $Y_i$'s. Results about this setting are not indicative about out-of-sample prediction, a crucial property in many tasks where least squares is routinely used, as well as in theoretical problems such as linear aggregration; see \cite{AudibertCatoni2011} for further discussion. 

The most basic method for minimizing $\ell$ from data -- the so-called ordinary least squares (OLS) estimator --  replaces the expectation in the definition of $\ell$ by an empirical average.
$$\betahat\in{\rm arg min}_{\beta\in\R^p}\,\ellhat(\beta)\mbox{, where } \ellhat(\beta)\equiv \frac{1}{n}\sum_{i=1}^n(Y_i-X_i^T\beta)^2.$$
This estimator is not hard to study when $n$ is large, $p$ is much smaller than $n$ and a linear model is assumed:
\begin{equation}\label{eq:linearmodel}\mbox{\bf Linear model: }Y = X^T\betamin + \epsilon,\mbox{ with }\left\{\begin{array}{l}(\epsilon,X)\mbox{ independent},\\ \Ex{\epsilon}=0\mbox{ and }\Ex{\epsilon^2}=\sigma^2<+\infty\end{array}\right..\end{equation}

Here we want to consider a completely model-free, non-parametric setting where no specific relationship between $X$ and $Y$ is assumed. Moreover, we want to allow for large $p$, with the only condition is that $p/n$ should be small. This rules out using classical asymptotic theory (which is not quantitative)  as well as Barry-Ess\'{e}en-type bounds (which do not work for $p\gg n^{2/3}$; see \cite{Bentkus2003} for the best known bounds). 

The theoretically optimal choice of $\beta$ that minimizes $\ell(\beta)$  is simply a vector $\betamin\in\R^p$ such that the coordinates of $X$ are $L_2$-orthogonal to $\epsilon\equiv Y-\betamin^TX$. This corresponds to the following generalization of \eqnref{linearmodel}.
\begin{equation}\label{eq:generalmodel}\mbox{\bf General model: }Y=X^T\betamin + \epsilon\mbox{ with }\left\{\begin{array}{l}\Ex{\epsilon^2},\Ex{|X|_2^2}<+\infty\\ \mbox{ and }\Ex{\epsilon\,X}=0.\end{array}\right.\end{equation}
Moreover, approximating the minimum loss $\ell(\betamin)$ corresponds to approximating $\betamin$ itself in the following sense:
\begin{equation}\label{eq:defsigma}\forall v\in\R^p\,:\,\ell(\betamin+v) = \ell(\betamin) + |\Sigma^{1/2}v|_2^2\mbox{ where }\Sigma\equiv \Ex{XX^T}.\end{equation}

\subsection{Our result, and previous work}

Here is the precise statement we prove.

\begin{theorem}[Proven in \secref{OLSproof}]\label{thm:OLS}Assume $\{(X,Y)\}\cup \{(X_i,Y_i)\}_{i=1}^n$ are as above and define $\betamin,\epsilon,Z$ and $\Sigma$ as in \eqnref{generalmodel} and \eqnref{defsigma}. Assume there exist $q\geq 2$, $1<\hyp,\hyp_*<+\infty$ and a positive semidefinite matrix $\Lambda\in \R^{p\times p}$ such that:
\begin{enumerate}
\item $\Ex{|X|_2^4}<+\infty$. Moreover, letting $\Sigma\equiv \Ex{XX^T}$, $\sqrt{\Ex{(v^TX)^4}}\leq \hyp\,v^T\Sigma\,v$ for all $v\in \R^p$.
\item Let $\Sigma^{-1/2}$ denote the Moore-Penrose pseudoinverse of $\Sigma^{1/2}$. Then the vector $Z\equiv \epsilon\,\Sigma^{-1/2}X$ satisfies $\Ex{(v^TZ)^2}\leq v^T\Lambda\,v$ and $\sqrt[q]{\Ex{(v^TZ)^{2q}}}\leq \hyp_*\,v^T\Lambda\,v$.\end{enumerate}
Choose $\delta,\eta,\eps\in(0,1)$ and assume:
$$n\geq\left( \frac{49}{\hyp^2\,\eps^2}\,(p+2\ln(6/\delta))\right)\vee \left(\frac{6^{2/q}\,(2+\eta)^2\,q^2\,(\hyp_*+1)}{\delta^{2/q}\,\eta^2}\right)$$  
Also define:
$$c(\eta) \equiv \frac{(2+\eta)\,(4+3\eta)}{4\eta}.$$
Then
$$\Pr{\ell(\betahat)-\ell(\betamin)\leq \frac{(1+\eta)\,\tr(\Lambda) + c(\eta)\,\lambda_{\max}(\Lambda)\ln(3/\delta)}{(1-\eps)^2\,n}}\geq 1-\delta.$$\end{theorem}
\thmref{OLS} implies
$$\Pr{\ell(\betahat)-\ell(\betamin)\leq (1+\liloh{1})\,\frac{\lambda_{\max}(\Lambda)\,p}{n}}=1-\left(\frac{\hyp_*}{n}\right)^{q/2-\liloh{1}}$$
whenever $\ln n = \liloh{p}$ and $p =\liloh{\frac{n}{\hyp^2}}$. This can be shown to be essentially sharp in the particular case of a linear model \eqnref{linearmodel} with Gaussian noise, where OLS satisfies $\ell(\betahat)-\ell(\betamin)\geq (1-\liloh{1})\,\lambda_{\max}(\Lambda)^2\,p/n$ with positive probability in that case, since $\lambda_{\max}(\Lambda)$ os simply the variance of the noise in this case.  

The proof of \thmref{OLS} consists of three steps. One is to use an explicit expression for OLS in order to express $\betahat-\betamin$. \thmref{main} is used to prove that a matrix that appears in the expression for this difference has bounded norm. The third step is to control the remaining expression, which is a sum of i.i.d. random vectors that we analyze via \lemref{appendix.vector} below.

Given the widespread use of OLS,  it seems surprising that all finite-sample results for it prior to 2011 were either considerably weaker (eg. did not bound $\ell(\betahat)-\ell(\betamin)$ directly) or required much stronger assumptions on the data generating mechanism; see \cite[Section 1]{AudibertCatoni2009} and \cite{HsuKakadeZhang2012} for more details on previous results. In the last two years of Audibert and Catoni \cite{AudibertCatoni2011} and Hsu et al. \cite{HsuKakadeZhang2012} both proved results related to our own \thmref{OLS} (below). However, our result is less restrictive in important ways. Hsu et al assumed i.i.d. subgaussian noise and bounded covariate vectors; moreover, they also need the condition $n\gg p\log p$, whereas our Theorem works for $n\gg p$ (assuming bounded $\hyp$ in both cases). The conditions of Audibert and Catoni are weaker but they assume $|v^TX|^2\leq B\,v^T\Sigma\,v$ uniformly for some constant $B>0$. It transpires from this brief discussion that \thmref{OLS} seems to be the first finite sample bound of optimal order that only assumes finitely many moments for $X$ and $Y$.

\begin{remark}Hsu et al \cite{HsuKakadeZhang2012} also derive finite sample performance bounds for ridge regression, a regularized version of OLS with an extra $\ell_2$ term. \thmref{main} and \thmref{OLS} can be adapted to that setting. Audibert and Catoni \cite{AudibertCatoni2011} also propose a ``robust"~least squares method based on a non-convex optimization problem, which we do not analyze here. It turns out, however, that this robust estimator depends on a quantity $\chi$ which is the same as our $\hyp$, so all computations in \cite[Section 3.2]{AudibertCatoni2011} are directly relevant to our setting.\end{remark}

\subsection{The proof}\label{sec:OLSproof}

\begin{proof}[of \thmref{OLS}] We will assume that $\Sigma$ has full rank; the general case follows from a simple perturbation argument. We also define $\Sigmahat$ as in \eqnref{defsigmahat}, that is,
$$\Sigmahat\equiv \frac{1}{n}\sum_{i=1}^nX_iX_i^T.$$
The assumptions on $X$ of \thmref{OLS} imply those of \thmref{main} (with $A_i=X_iX_i^T$). Tis implies that the event
\begin{equation}{\sf Lower}\equiv \left\{\forall v\in\R^p\,:\, v^T\Sigmahat\,v\geq \left(1 - \eps\right)\,v^T\Sigma v\right\}\end{equation}
satisfies $\Pr{{\sf Lower}}\geq 1-\delta/3$ whenever the condition on $n$ in \thmref{OLS} is satisfied.

Also define
\begin{equation}\label{eq:defzi}\epsilon_i \equiv Y_i - X_i^T\betamin\mbox{ and }Z_i\equiv \epsilon_i\,\Sigma^{-1/2}X_i,\,i=1,2,3,\dots,n.\end{equation}
The $Z_i$ are independent vectors whose law is the same as that of $Z$ in \thmref{OLS}. This implies that the following Lemma may be applied. 
\begin{lemma}[Proven in \secref{vector.appendix}]\label{lem:appendix.vector} Suppose $Z_1,\dots,Z_n\in\R^p$ are i.i.d. random vectors whose coordinates have finite $2q$ moments for some $q\geq 2$. Assume $\Lambda\in\R^{p\times p}$ is a positive semidefinite matrix and $\hyp_*>0$ are such that 
$$\forall v\in\R^p\,:\, \Ex{(v^TZ_1)^2}\leq v^T\Lambda\,v\mbox{ and }\sqrt[q]{\Ex{(v^TZ_1)^{2q}}}\leq \hyp_*\,v^T\Lambda\,v.$$
Then for any $\eta\in(0,1/2)$:
$$\Pr{\frac{|\sum_{j=1}^nZ_j|_2}{\sqrt{n}}\leq \,\sqrt{(1+\eta)\tr(\Lambda) + c(\eta)\,\lambda_{\max}(\Lambda)\,t}}\geq 1 - e^{-t}-2\left(\frac{d_q(\eta)(\hyp_{*}+1)}{n}\right)^{\frac{q}{2}}$$
where $$c(\eta) \equiv \frac{(2+\eta)\,(4+3\eta)}{4\eta}\mbox{ and }d_q(\eta)\equiv \frac{(2+\eta)^2\,q^2}{\eta^2}.$$
\end{lemma}
The Lemma implies that the event {\sf Vector} defined below,
$${\sf Vector}\equiv \left\{\frac{|\sum_{j=1}^nZ_j|_2}{\sqrt{n}}\leq \,\sqrt{(1+\eta)\tr(\Lambda) + c(\eta)\,\lambda_{\max}(\Lambda)\,\ln(3/\delta)}\right\},$$satisfies
$$\Pr{{\sf Lower}\cap {\sf Vector}}\geq 1- 2\delta/3 - 2\left(\frac{d_q(\eta)(\hyp_{*}+1)}{n}\right)^{\frac{q}{2}}\geq 1-\delta$$
by our assumptions on $n$.

From now on we analyze OLS conditionally on {\sf Lower}$\cap${\sf Vector}. Notice that $\Sigmahat$ is invertible: this is because $\Sigma$ is invertible and ${\sf Lower}$ holds. A simple calculation shows that $\betahat$ -- the minimizer of $\ellhat(\cdot)$ -- can be written as:
\begin{eqnarray*}\betahat &= &\Sigmahat^{-1}\,\left\{\frac{1}{n}\sum_{i=1}^nY_iX_i\right\}\\ 
\mbox{(use \eqnref{defzi})} &=&\Sigmahat^{-1}\left\{\frac{1}{n}\sum_{i=1}^nX_i(X_i^T\betamin)\right\} +  \Sigmahat^{-1}\left\{\frac{1}{n}\sum_{i=1}^n\epsilon_i\,X_i\right\}\\ &=& \betamin +(\Sigmahat^{-1}\Sigma^{1/2})\,\left\{\frac{1}{n}\sum_{i=1}^nZ_i\right\}.\end{eqnarray*}
Going back to \eqnref{defsigma} we obtain:
\begin{equation}\label{eq:excesslossveryend}\ell(\betahat)-\ell(\betamin) = \left|(\Sigma^{1/2}\Sigmahat^{-1}\Sigma^{1/2})\,\left\{\frac{1}{n}\sum_{i=1}^nZ_i\right\}\right|_2^2\leq \left|\Sigma^{1/2}\Sigmahat^{-1}\Sigma^{1/2}\right|_{2\to 2}^2\,\left|\frac{1}{n}\sum_{i=1}^nZ_i\right|_2^2.\end{equation}
The two rightmost terms in the previous display are bounded via {\sf Lower} and {\sf Vector}. To see this we begin by applying \eqnref{norminverse} above:
\begin{eqnarray*}\left|\Sigma^{1/2}\Sigmahat^{-1}\Sigma^{1/2}\right|^2_{2\to 2} &=& \left(\inf_{v\in\R^p,|v|_2=1}v^T\Sigma^{-1/2}\Sigmahat\Sigma^{-1/2}\,v\right)^{-2}\\ 
\mbox{(take $w=\Sigma^{-1/2}v$)} &=& \left(\inf_{w\in\R^p,w^T\Sigma w=1}w^T\Sigmahat\,w\right)^{-2}\\ \mbox{(use {\sf Lower})}&\leq &\left(\frac{1}{1-\eps}\right)^2\end{eqnarray*}
and
$$\left|\frac{1}{n}\sum_{i=1}^nZ_i\right|_2^2\leq \frac{(1+\eta)\tr(\Lambda) + c(\eta)\,\lambda_{\max}(\Lambda)\,\ln(2/\delta)}{n}\,\,\mbox{(by {\sf Vector}).}$$
Plugging these bounds into \eqnref{excesslossveryend} results in 
$$\ell(\betahat)-\ell(\betamin)\leq \frac{(1+\eta)\tr(\Lambda) + c(\eta)\,\lambda_{\max}(\Lambda)\,\ln(2/\delta)}{(1-\eps)^2\,n},$$
and this inequality holds whenever ${\sf Lower}\cap {\sf Vector}$ occurs. In particular, the probability of the last display satisfies the bound claimed in the Theorem.\end{proof}

\subsection{Proof of the auxiliary result on sums of random vectors}\label{sec:vector.appendix}

\begin{proof}[of \lemref{appendix.vector}]
Write $S_0=0$ and $S_i = S_{i-1}+Z_i$, $1\leq i\leq n$. We note that:
$$|S_n|^2 = \sum_{i=1}^n |Z_i|^2 + M_n$$
where $M_0=0$ and
$$M_i\equiv \sum_{j=1}^i\ip{S_{i-1}}{Z_i}\,(1\leq i\leq n)$$
is a martingale with respect to the filtration $\sF_0=\{\emptyset,\Omega\}$, 
$$\sF_i\equiv \sigma(Z_1,\dots,Z_i)\,(1\leq i\leq n).$$
Now define $h_i\equiv S_i/|S_i|$ if $|S_i|\neq 0$, and $h_i=0$ otherwise. The following random variable will be important later on. 
$$V_i\equiv \sum_{j=1}^i\,({h^T_{j-1}}{Z_j})^2.$$

We will use the following estimates (proven subsequently).

\begin{claim}\label{claim:U}For any $\alpha>0$, $$\Pr{\sum_{j=1}^n|Z_i|_2^2\geq (1+\alpha)\,n\tr(\Lambda)}\leq \left(\frac{q^2(\hyp_{*}+1)}{\alpha^2n}\right)^{q/2}$$ and $$\Pr{V_n> (1+\alpha)n\lambda_{\max}(\Lambda)}\leq \left(\frac{q^2(\hyp_{*}+1)}{\alpha^2n}\right)^{q/2}.$$.\end{claim}
We will also use the following simple fact about martingales, which we prove in the appendix
\begin{proposition}[Proven in \secref{Appmartingale}]\label{prop:Appmartingale}Suppose $\{N_i\}_{i=0}^n$ is a square-integrable martingale with respect to a filtration $\{\sG_i\}_{i=0}^n$. Define $W_0=0$ and
$$W_i\equiv \sum_{j=1}^i\,(\Ex{(N_{j}-N_{j-1})^2\mid\sG_j} + (N_j-N_{j-1})^2)\,\,(1\leq i\leq n).$$
Then for any $\xi>0$ and $t\geq 0$,
$$\Pr{\exists 1\leq i\leq n\,:\, N_i>\frac{\xi}{2}\,W_i + \frac{t}{\xi}}\leq e^{-t}.$$\end{proposition}

We will apply this to $N_i=M_i$ and $\sG_i=\sF_i$. Since $\Ex{(Z_i^Tv)^2}\leq v^TCv$ for all $v$,
$$W_i = \sum_{j=1}^i (\ip{S_{j-1}}{Z_j}^2 + |\Lambda^{1/2}S_{j-1}|_2^2)\leq (V_n+n\lambda_{\max}(\Lambda))\,\max_{1\leq j\leq i}|S_{j-1}|_2^2.$$
We conclude that, for any $\xi>0$,
$$\Pr{\exists 1\leq i\leq n\,:\, M_i>\frac{\xi}{2}\,(V_n+\lambda_{\max}(\Lambda)) \,\max_{1\leq j\leq i}|S_{j-1}|_2^2+ \frac{t}{\xi}}\leq e^{-t}.$$
Combining this with \claimref{U} and the definition of $V_n$ shows that, for any choice of $\xi>0$, $0<\alpha<1$, we have:

\begin{equation}\label{eq:highprobvector}\Pr{\forall 1\leq i\leq n\,:\,|S_i|_2^2 \leq \left(\begin{array}{l}(1+\alpha)\,n\,\tr(\Lambda)\\ +\frac{\xi}{2}\,(2+\alpha)\,n\lambda_{\max}(\Lambda)) \,\max_{1\leq j\leq i}|S_{j-1}|^2+ \frac{t}{\xi}\end{array}\right)}\geq 1-\delta,\end{equation}
where
$$\delta \equiv e^{-t} + 2\left(\frac{q^2(\hyp_{*}+1)}{\alpha^2n}\right)^{q/2}.$$
Now fix some $\alpha$, make the choice of 
$$\xi\equiv \frac{2\alpha}{(2+\alpha)\,n\lambda(\Lambda)}$$
and apply \eqnref{highprobvector} to the value $i_*\in\{1,\dots,n\}$ achieving the maximum of $|S_i|$. We have that, with probability $\geq 1-\delta$,
$$|S_{i_*}|^2\leq (1+\alpha)\,n\,\tr(\Lambda)  + \alpha\,|S_{i_*}|^2 + \left(\frac{2+\alpha}{2\alpha}\right)\,n\,\lambda_{\max}(\Lambda)\,t,$$
which implies that 
$$|S_n|^2\leq |S_{i_*}|^2\leq \left(\frac{1+\alpha}{1-\alpha}\right)\,n\,\tr(\Lambda) +  \left(\frac{2+\alpha}{2\alpha(1-\alpha)}\right)\,n\,\lambda_{\max}(\Lambda)\,t$$
with probability $\geq 1-\delta.$ This is precisely the desired result once we choose:
$$\alpha \equiv \frac{\eta}{2-\eta}.$$
\end{proof}

To finish, we must now prove \claimref{U}.
\begin{proof}[of \claimref{U}]~We prove the second (harder) assertion first. Note that $$V_i-\sum_{j=1}^i|\Lambda^{1/2}h_{j-1}|^2\,(1\leq i\leq n)$$ is a martingale with respect to the filtration $\{\sF_i\}_{i=1}^n$.  The Burkholder-Davis-Gundy inequality \eqnref{BDG} implies, for any $q\geq 2$,
\begin{eqnarray}\nonumber \frac{\Ex{|V_n - \sum_{j=1}^n|\Lambda^{1/2}h_{j-1}||^q}^{1/q}}{q\sqrt{n}}\,&\leq& \Ex{\left(\frac{1}{n}\sum_{j=1}^n(\ip{h_{j-1}}{Z_j}^2-|\Lambda^{1/2}h_j|^2)^2\right)^{q/2}}^{1/q} \\ \nonumber \mbox{(convexity)}&\leq &\left( \frac{1}{n}\sum_{i=1}^n\Ex{|\ip{h_{j-1}}{Z_j}^2-|\Lambda^{1/2}h_j|^2|^{q}}^{1/q}\right)^{1/2}\\
\nonumber \mbox{(Minkovski)}&\leq & \left(\frac{1}{n}\sum_{J=1}^n\left(\Ex{|\ip{h_{j-1}}{Z_j}|^{2q}}^{1/q} +|\Lambda^{1/2}h_{j-1}|^2\right)\right)^{1/2}\\ \nonumber \mbox{(defn. of $\hyp_{*}$)} &\leq & \left(\max_{1\leq j\leq N}(\hyp_{*} +1)\,|\Lambda^{1/2}h_{j-1}|^2\right)^{1/2}\\
\nonumber \mbox{($|h_{j-1}|\leq 1$)} &\leq& \sqrt{(\hyp_{*}+1)}\,\lambda_{\max}(\Lambda)\,.\end{eqnarray}
Using again that $|h_{j-1}|\leq 1$ always, $\sum_{j=1}^n|\Lambda^{1/2}h_{j-1}|\leq n\lambda_{\max}(\Lambda)$. We deduce:
\begin{eqnarray*}\Pr{V_n\geq (1+\alpha)\,n\lambda_{\max}(\Lambda)}&\leq &\Pr{\frac{V_n - \sum_{j=1}^n|C^{1/2}h_{j-1}|}{\alpha n\lambda_{\max}(\Lambda)}\geq 1}\\ &\leq &\frac{\Ex{(V_n - \sum_{j=1}^n|C^{1/2}h_{j-1}|)^q}}{(\alpha n)^q\lambda_{\max}(\Lambda)^q}\\ &\leq &\left(\frac{q^2(\hyp_{*}+1)}{\alpha^2n}\right)^{q/2}.\end{eqnarray*}
In order to prove the assertion about $\sum_i|Z_i|^2$, we note that $\Ex{|Z_i|^2}=\tr(\Lambda)$. Using the fact that centered sums of independent random variables are also martingales, we may apply the BDG inequality \eqnref{BDG} again to deduce:
$$\Ex{|\sum_{j=1}^n (|Z_j|^2-\tr(\Lambda))|^q}^{1/q}\leq \sqrt{n}q\,(\Ex{|Z_1|^{2q}}^{1/q}+\tr(\Lambda))^{1/2}.$$
Applying \lemref{powertrace} in \secref{powertrace} to $A=Z_1Z_1^T$ gives $\Ex{|Z_1|^{2q}}\leq \hyp^{q}_{*}\,\tr(\Lambda)^q$, and we may prove the tail bound on $\sum_j|Z_j|^2$ like we proved the bound for $V_n$.[End of proof of \claimref{U}]\end{proof}

\section{Restricted eigenvalues in high dimensions}\label{sec:re}

\subsection{Setup}\label{sec:resetup}

Our second application of \thmref{main} is to the general areas of Compressed Sensing and High Dimensional Statistics. The basic problem
for these two areas is to recover a vector $\betamin$ from a set of pairs $(x_1,Y_1),\dots,(x_n,Y_n)\in\R^p\times \R$, which are assumed to satisfy
\begin{equation}\label{eq:fixeddesign}Y_i = x_i^T\betamin + \epsilon_i\end{equation}
where $\epsilon_1,\dots,\epsilon_n$ represent some kind of noise and -- most importantly -- the dimension $p$ may greatly exceed the number $n$ of measurements. The aforementioned fields tend to interpret this setup in different ways. Whereas in Compressed Sensing one tends to think of the $x_i$'s as measurement vectors as controlled by the ``experimenter", for a statistician the $x_i$ and $Y_i$ are generated by a random process that is not under control (and the whole problem corresponds to linear regression $p\gg n$; \secref{linearsparse} below). 

It should be clear that, given $p\gg n$, the above problem is severely underdetermined. However, {\em sparsity} may be used  as a key enabling assumption. It is known that if the vector $\betamin$ has $s\ll n/\log p$ non-zero coordinates, then it may be recovered up to error of the order $\sigma^2s\log p/n$. This is only $\bigoh{\log p}$ times larger than the error of OLS which ``knows"~the support of $\betamin$. Most importantly, there are {\em computationally efficient} estimators achieving this rate. These developments and their extensions comprise a vast literature which we will not try to survey; we refer instead to a recent book \cite{BuhlmannVanDerGeer2011} and a handful of important papers \cite{CandesTao2007,Donoho2006,BickelRT2009,CandesRT2006,MeinshausenYu2009} for more information on these topics.

Computationally efficient estimators achieving this rate require certain conditions besides sparsity. Denote by $\bXhat$ the design matrix:
\begin{equation}\label{eq:designmatrix}\bXhat\equiv \frac{1}{n}\sum_{i=1}^nx_ix_i^T.\end{equation}
Several sufficient conditions on $\bXhat$ are known to ensure the fast rates we have described, including uniform uncertainty principles, restricted isometry, sparse eigenvalues and incoherence; see eg. \cite{CandesTao2007,BickelRT2009,CaiWX2010} and especially the paper \cite{BuhlmannVanDerGeer2008} where these conditions are compared. In this paper we focus on so-called restricted eigenvalue conditions, which are amongst the least restrictive in this class.  

\begin{definition}[Restricted eigenvalues; \cite{BuhlmannVanDerGeer2008,BickelRT2009}]Let $A\in\R^{p\times p}$. Let $S\subset \{1,\dots,p\}$ be a non-empty subset and $\alpha>0$. We define the set:
$$\sC(S,\alpha)\equiv \{v\in \R^p\,:\,|v_{S^c}|_1\leq \alpha\,|v_S|_1\}.$$
(Here $v_S$ denotes the restriction of $v$ to $S$, cf. \secref{prelim}.)
The restricted eigenvalue constant for $(A,S,\alpha)$, denoted by $\re(A,S,\alpha)$, is the largest value of $R>0$ such that:
$$\forall v\in \sC(S,\alpha)\,:\, R^2\,|v_S|^2_2\leq v^T A v.$$
Moreover, $\re(A,s,\alpha)$ is the minimum of $\re(A,S,\alpha)$ over $S\subset \{1,\dots,p\}$ with $|S|=s$.\end{definition}

In the setting of \eqnref{fixeddesign} one may take $S$ as the support of $\betamin$. Assuming $\re(\bXhat,S,\alpha)$ is bounded for some specific $\alpha>0$ ensures that estimators such as the Dantzig selector \cite{CandesTao2007,BickelRT2009} and the LASSO \cite{BuhlmannVanDerGeer2008,BickelRT2009} will achieve the near-OLS error rate defined above. Here is one example by Buhlmann and Van der Geer \cite{BuhlmannVanDerGeer2008} which may be applied to a fixed-design linear regression model

\begin{theorem}\label{thm:LASSO}Assume the fixed design linear model in \eqnref{fixeddesign} where the vectors $x_i\in\R^p$ are deterministic and the noise terms $\epsilon_i$ are independent random variables with Gaussian distribution, $\Ex{\epsilon_i}=0$  and $\Ex{\epsilon_i^2}=\sigma^2>0$. Assume further that $\betamin$ is supported on a subset $S\subset \{1,\dots,p\}$ of size $s$. Finally, suppose that the design matrix $\bXhat$ has diagonal entries equal to $1$ and $\re(\bXhat,S,3)>0$. Consider the LASSO estimator:
$$\betahat[\lambda]\equiv {\rm arg min}\left\{\frac{1}{n}\sum_{i=1}^n(x_i^T\beta - Y_i)^2 + \lambda\,|\beta|_1\,:\, \beta\in\R^p\right\}.$$
Then there exists a choice of $\lambda=\lambda(\sigma^2,n,p)$ such that, with probability $\geq 1-p^{-2}$:

\begin{equation}\label{eq:fastrateexample}\betahat[\lambda]-\betamin\in\sC(S,3)\mbox{ and }|\bXhat^{1/2}(\betahat-\betamin)|_2^2\le c\,\sigma^2\,\frac{s\log p}{n}\end{equation}
where $c>0$ depends only on $\re(\bXhat,S,3)$.\end{theorem}

We emphasize that this estimator has performance which nearly matches that of OLS when the support of $\betamin$ is known. Similar results could be achieved by trying all potential supports: the merit of the LASSO and related methods is computational efficientcy. 

We note in passing that there is also a fairly sizable literature on how well the LASSO and other methods do when $\betamin$ is only approximately sparse and a linear model is not necessarily valid. We will mostly refrain from discussing this in what follows, and refer to \cite{BickelRT2009,BelloniEtAl2011} for further discussion of this topic.   

\subsection{Our result, and related work}

In Statistics and Machine Learning it is natural to assume that the vectors $x_i$ are generated randomly by a mechanism that is not under control of the experimenter. One may enquire whether such random ensembles will typically satisfy restricted eigenvaue properties. This problem has been addressed for Gaussian ensembles by Raskutti et al. \cite{RaskuttiEtAl2010} and for subgaussian and bounded-coordinate ensembles by Rudelson and Zhou \cite{RudelsonZhou2013}. In both cases it is shown that $\re({\bXhat},s,\alpha)$ can be bounded in terms of $\re(\Ex{\bXhat},s,\tilde{\alpha})$  for some $\tilde{\alpha}\approx \alpha$, when $s\gg n/\log p$. We prove here that finite moment assumptions suffice to bound restricted eigenvalues of chosen sets $S$.  \footnote{Note that the bounded coordinate case neigher implies nor is implied by our results.}

\begin{theorem}\label{thm:re}Let $X_1,\dots,X_n\in\R^p$ be independent and identically distributed random vectors whose coordinates have $2q$ moments for some $q>2$. Define $\Sigma$ and $\Sigmahat$ as in the Introduction, and assume that $\hyp,\hyp_*\in(1,+\infty)$ are such that
$$\forall v\in\R^p\,: |v|_0\leq n\Rightarrow \sqrt{\Ex{(v^TX_1)^4}}\leq \hyp\,v^T\Sigma\,v$$
and
$$\forall 1\leq j\leq p\,:\, \Ex{X_1[j]^{2q}}^{\frac{1}{q}}\leq \hyp_*\,\Ex{X_1^2[j]}.$$
Define diagonal matrices $\widehat{D}_{2,n}$ and $D_2$ corresponding to the diagonals of $\Sigmahat$ and $\Sigma$ (respectively). Set 
$$\bXhat\equiv \widehat{D}_{2,n}^{-1/2}\Sigmahat\,\widehat{D}_{2,n}^{-1/2} \mbox{(this is the design matrix of }x_i\equiv \widehat{D}_{2,n}^{-1/2}X_i\mbox{, $1\leq i\leq n$)}$$
and $\bX\equiv D_2^{-1/2}\Sigma D_2^{-1/2}$ with the convention that the $(j,j)$th entry of $\widehat{D}_{2,n}^{-1/2}$ (resp. $D_2^{-1/2}$) is zero whenever the corresponding entry of $\widehat{D}_{2,n}$ (resp. $D_2$) is zero. Assume that $\delta,\eps\in(0,1/2)$ and $S\subset \{1,\dots,p\}$ with cardinality $|S|=s$, and set$$\tilde{\alpha}=\alpha\sqrt{\frac{1+\eps}{1-\eps}}\mbox{ and }C\equiv 784\,[(1+\eps)\,(1+\alpha)^2\,]\,\hyp^2.$$
Finally, assume
$$n\geq \max\left\{\left(\frac{C\,(1+2\ln(p/4\delta))}{\re(\bX,S,\tilde{\alpha})^2\,\eps^4}\right)\,s,\frac{4q^2\,3^{2/q}s^{2/q}}{\delta^{2/q}}\right\}$$
Then the following three properties hold simultaneously with probability $\geq 1-\delta$.
\begin{itemize}
\item [{\bf C1}] Let $x\in \R^p$ be any vector such that $\widehat{D}_{2,n}x\in\sC(S,\alpha)$. Then $D_2x\in \sC(S,\tilde{\alpha})$.
\item [{\bf C2}] For any $x$ as above, $x^T\Sigmahat x\geq (1-\eps)^2\,x^T\Sigma\,x.$
\item [{\bf C3}] $\re(\bXhat,S,\alpha)\geq (1-\eps)\,\re(\bX,S,\tilde{\alpha})$.  
\end{itemize}\end{theorem}
The upshot (valid for constant $\re(\bX,S,\tilde{\alpha})$, $\hyp\,\hyp_*$) is that the restricted eigenvalue property holds with high probability whenever $\bX$ has this property and $s=\liloh{n/\ln p}$ \footnote{A slightly tighter calculation shows that $s=\liloh{n/\max\{1,\ln(ep/n)\}}$ would still suffice.}.

Let us note the main differences between this theorem and the results in \cite{RaskuttiEtAl2010,RudelsonZhou2013}: our theorem holds for a specific choice of $S\subset \{1,\dots,p\}$ -- ie. it is {\em not} uniform over $S$ with $|S|=s$ -- and uses the ``normalized"~matrix $\bXhat$ instead of $\Sigmahat$. Both differences are related to our moment assumptions, and both turn out not to be problematic in certain scenarios, such as  ``randomized, RIPless compressed sensing"~\cite{CandesP2011} and statistical regression problems, where one wants to solve one problem instance and uniform guarantees are unnecessary (cf. \secref{linearsparse} below). We note that the normalization on $\Sigmahat$ is farly natural, at it ensures the ``unit diagonal" condition in \thmref{LASSO}. We also note that stronger moment assumptions allow for stronger conclusions via the same proof methods; we illustrate this with a simple example.

\begin{example}[Randomness efficient CS matrices]\label{ex:efficient} Suppose $X_1,\dots,X_n\in \R^p$ have four-wise independent coordinates which are uniform in $\{-1,+1\}$. This ensemble clearly satisfies the assumptions of the Theorem, with $\widehat{D}_{2,n}=D_2=I_{p\times p}$. Inspection of the proof implies that there exists some $C,c=c(\alpha)>0$ independent of $p$ or $s$ such that $n\geq C\,s\log p$ implies that
$\re(s,\Sigmahat,\alpha)>c$ with probability $\geq 1-p^{-2}$. Since one can sample a four-wise independent vector using $\bigoh{\log p}$ bits, this implies that one may construct a matrix with positive restricted eigenvalues of order $s$ using $\bigoh{s\log^2p}$ bits.\end{example}

\subsection{A digression on linear regression with random design}\label{sec:linearsparse}

\thmref{re} is quite obviously applicable to fixed design regression as in \thmref{LASSO}. As it turns out, it may also be applied to the random design setting discussed in \secref{OLSsetup} when the dimension $p$ is much greater than the number of samples $n$. For simplicity we will focus on how this is done in the {\em linear model} setting \eqnref{linearmodel}, in which case we may apply \thmref{LASSO} directly; a {\em general model} analysis would require ideas from \cite{BelloniEtAl2011}.

Let $(X_1,Y_1),\dots,(X_n,Y_n)$ be i.i.d. copies of a random pair $(X,Y)\in\R^p\times \R$ that satisfies \eqnref{linearmodel} where each $\epsilon_i$ is mean-zero Gaussian with variance $\sigma^2$. Assume the conditions of \thmref{re}, and apply \thmref{LASSO} with a ``change of variables" where each $X_i$ is replaced with $x_i=\widehat{D}_{2,n}^{-1/2}X_i$. This has the effect of making the diagonal elements of $\bXhat$ equal to $1$, as required by \thmref{LASSO}. Note also that this change of variables consists of replacing  $\beta$ with $D_{2,n}^{1/2}\beta$ in the LASSO estimator in \thmref{LASSO}. Combining this with \thmref{re} gives:
$$|\Sigma^{1/2}(\betahat-\betamin)|_2^2\leq c_1\,|\Sigmahat^{1/2}(\betahat-\betamin)|_2^2 = c_1\,|\bXhat^{1/2}D_{2,n}^{1/2}(\betahat-\betamin)|_2^2\leq c_2\,\sigma^2\,\frac{s\log p}{n}$$
with probability $1-\bigoh{p^{-2}}$, as long as $\re(\bX,S,c)>0$ for some $c>3$ and the other parameters are chosen in their proper ranges.  

\subsection{Proof ideas, and the transfer principle}

Besides \thmref{main}, the key element in the proof of \thmref{re} is a very simple ``transfer lemma" (given below) that shows that this control implies lower tail bounds for $x^T\Sigmahat\,x$ for all $x$, at the cost of an extra $\ell_1$ term in the lower bound. 

\begin{lemma}[Transfer Principle; proven below]\label{lem:transfer} Suppose $\Sigmahat$ and $\Sigma$ are matrices with non-negative diagonal entries, and assume $\eta\in(0,1)$, $d\in\{1,\dots,p\}$ are such that 
$$\forall v\in \R^p\mbox{ with }|v|_0\leq d,\, v^T\Sigmahat\,v\geq (1-\eta)\,v^T\Sigma\,v.$$
Assume $D$ is a diagonal matrix whose elements $D[j,j]$ are non-negative and satisfy $D[j,j]\geq \Sigmahat[j,j]-(1-\eta)\,\Sigma[j,j]$. Then 
$$\forall x\in\R^p,\, x^T\Sigmahat x\geq (1-\eta)\,x^T\Sigma\,x - \frac{|D^{1/2}x|_1^2}{d-1}.$$\end{lemma}

Raskutti et al. \cite{RaskuttiEtAl2010} prove such a bound directly for Gaussian ensembles, and note that it implies the restricted eigenvalue property when the population design matrix has this property. In our case we use \thmref{main} to control of $\Sigmahat$ over sparse vectors, and combine it with this Lemma to obtain the appropriate control over the cone $\sC(S,\alpha)$. As noted in the introduction, this Transfer Principle implies a version of the main result of Rudelson and Zhou \cite{RudelsonZhou2013}; see Appendix \ref{sec:appendix.rudelson} for details. \\

\begin{proof}[of \lemref{transfer}] We assume $D$ is invertible; the general case follows via a simple continuity argument. We also set 
$$A\equiv D^{-1/2}(\Sigmahat - (1-\eps)\,\Sigma)D^{-1/2}.$$
Notice that $v^T Av\geq 0$ for all $d$-sparse vectors, and also that $0\leq A[j,j]\leq 1$ for each $1\leq j\leq p$. We will prove that:
$$(\star)\,\,\forall y\in\R^p\,:\, y^T A y\geq -\frac{|y|_1^2}{d-1},$$
which implies the Lemma once we set $y=D^{1/2}x$. 

To prove $(\star)$ we use a probabilistic argument related to Maurey's empirical method. We may write 
$$y = \sum_{j=1}^p\,|y|_1\,p_j\,s_j\,e_j$$
where $s_j\in \{-1,+1\}$ is the sign of $y[j]$ and $p_j\equiv |y[j]|/|y|_1$. 

The $p_j$'s are non-negative and sum to $1$. Therefore we may define $v_1,\dots,v_d$ to be independent and identically distributed random vectors with the following distribution:
$$\forall 1\leq j\leq p\,:\,\Pr{v_1=|y|_1\,s_j\,e_j}=p_j.$$
Note that $\Ex{v_1}=\dots=\Ex{v_d}=y$. The vector 
$$v = \frac{1}{d}\sum_{i=1}^dv_i$$
has at most $d$ nonzero coordinates, so
$v^T A v\geq 0$. Taking expectations, we see that:
\begin{equation}\label{eq:expectation}\Ex{v^T A v} = \frac{1}{d^2}\sum_{i,r=1}^d\Ex{v_i^T A v_r}\geq 0.\end{equation}
It remains to compute this expectation. When $i\neq r$, $v_i$ and $v_r$ are independent and:
$$\sum_{i\neq r}\Ex{v_i^T Av_r} = d(d-1)\,\Ex{v_i}^T A\Ex{v_r} = d(d-1)\,y^T A y.$$
When $i=r$ and $v_i=|y|_1\,s_j\,e_j$ we see that $v_i^T A v_i = |y|_1^2 A[j,j]\leq |y|_1^2$ because $A[j,j]\leq 1$ for each $j$. Thus
$$\sum_i\,\Ex{v_i^T A v_i}\leq d\,|y|_1^2.$$
Combining the two previous displays with \eqnref{expectation} gives
$$\left(1-\frac{1}{d}\right)\,y^T A y + \frac{1}{d}\,|y|_1^2\geq 0,$$
which is the same as ($\star$).\end{proof}

\subsection{Proof}

In this section we prove \thmref{re}. The first step is where we use \thmref{main} and \lemref{transfer}.

\begin{lemma}\label{lem:intermediate}Under the assumptions of \thmref{re}, the following event holds with probability $\geq 1-\delta/3$:
$${\sf Lower}\equiv \left\{\forall v\in \R^p\,:\, v^T\Sigmahat v\geq \left(1-\frac{\eps}{2}\right)\,v^T\Sigma v - \frac{392\,\hyp^2\,(1+2\ln(p/4\delta))}{n}\,|\widehat{D}_{2,n}^{1/2}v|_1^2\right\}.$$\end{lemma}
\begin{remark}This Lemma does not require positive restricted eigenvalues, and may be used to analyze the LASSO as a $\ell_1$ penalized method. See the final section for more on this.\end{remark}

\begin{proof}Define:
$$d\equiv \left\lfloor \frac{\eps^2\,n}{196\,\hyp^2\,(1+2\ln(p/4\delta))}\right\rfloor.$$
Note that 
\begin{equation}\label{eq:dgeq4} \frac{\eps^2\,n}{196\,\hyp^2\,(1+2\ln(p/4\delta))}\geq 4\mbox{ and }7\hyp\,\sqrt{\frac{d + 2\ln(p^d/4\delta)}{n}}\leq 7\hyp\,\sqrt{\frac{d (1+ 2\ln(p/4\delta))}{n}}\leq \frac{\eps}{2}\end{equation}

 by the assumptions of \thmref{re}. We begin by proving that:
\begin{equation}\label{eq:goallemintermediate}\mbox{\bf First goal: } \Pr{\forall x\in\R^p\,:\, |x|_0\leq d\Rightarrow  x^T\Sigmahat x\geq (1-\eps/2)\,x^T\Sigma\,x}\geq 1-\delta/3.\end{equation}
To do this we consider the complement of the event above. Given $U\subset 
\{1,2,3,\dots,p\}$ we denote by $\R^U$ the set of vectors supported on $U$ (recall the definition of support in \secref{prelim}.
$$\R^U\equiv\{x\in\R^p\,:\, {\sf supp}(x)\subset U\}.$$
We may reformulate our goal as
\begin{equation}\label{eq:goal2lemintermediate}\mbox{\bf First goal: } \Pr{\bigcup_{U\subset\{1,\dots,p\}\,:\, |U|=d}\left\{\exists x\in\R^U\,:\, x^T\Sigmahat\,x< (1-\eps/2)\,x^T\Sigmahat\,x \right\}}\leq \frac{\delta}{3}.\end{equation}
We will use a union bound. Note that each $\R^U$ is a isometric copy of $\R^d$. One may project the $X_i$ to $\R^U$ and then apply \thmref{main} to those projected vectors to obtain: $$\forall U\subset \{1,\dots,p\},\,|U|=d\,:\, \Pr{\exists x\in\R^U\,:\, x^T\Sigmahat\,x\leq (1-\eps/2)x^T\Sigma\,x}\leq \frac{8\delta}{p^d}.$$
More specifically, this follows from \thmref{main} with $d$ replacing $p$ and $8\delta/p^d$ replacing $\delta$. Notice that with these choices 
$$7\hyp\,\sqrt{\frac{d + 2\ln(p^d/4\delta)}{n}} \leq  7\hyp\,\sqrt{\frac{d (1+ 2\ln(p/4\delta))}{n}}\leq 7\hyp\,\sqrt{\frac{\eps^2 n/196\hyp^2}{n}}=\frac{\eps}{2}.$$
Plugging this into \eqnref{goal2lemintermediate} and applying a union bound over $U$ gives:
$$\Pr{\bigcup_{U\subset\{1,\dots,p\}\,:\, |U|=d}\left\{\exists x\in\R^U\,:\, x^T\Sigmahat\,x< (1-\eps/2)\,x^T\Sigma\,x\right\}}\leq 2\,\binom{p}{d}\,\frac{8\delta}{p^d}.$$
Now note that
$$\binom{p}{d}\leq \frac{p^d}{d!}\leq \frac{p^d}{24}\mbox{ since }d\geq 4.$$
Therefore,
$$\Pr{\bigcup_{U\subset\{1,\dots,p\}\,:\, |U|=d}\left\{\exists x\in\R^U\,:\, x^T\Sigmahat\,x<\left(1-\frac{\eps}{2}\right)\,x^T\Sigma\,x\right\}}\leq \frac{\delta}{3}.$$
This gives our first goal \eqnref{goallemintermediate}. To obtain the Lemma from this, we apply the transfer principle (\lemref{transfer}) whenever {\sf Lower} holds, using \eqnref{dgeq4} to deduce$$d-1\geq \frac{\eps^2\,n}{196\,\hyp^2\,(1+2\ln(p/4\delta))} - 2 \geq \frac{\eps^2\,n}{392\,\hyp^2\,(1+2\ln(p/4\delta))}$$\end{proof}

We now present the proof of \thmref{re}.
\begin{proof}[of \thmref{re}] We define {\sf Lower} as in  \lemref{intermediate}, and consider two other events: 
\begin{eqnarray*}{\sf Diag}_-&\equiv &\{\forall 1\leq j\leq p\,:\,\widehat{D}_{2,n}[j,j]\geq (1-\eps)\,D_{2}[j,j];\\
{\sf Diag}_+&\equiv & \{\forall j\in S\,:\,\widehat{D}_{2,n}[j,j]\leq (1+\eps)\,D_{2}[j,j]\}.\end{eqnarray*}

We claim that conclusions {\bf C1}, {\bf C2} and {\bf C3} in \thmref{re} hold whenever {\sf Lower}, ${\sf Diag}_+$ and ${\sf Diag}_-$ all hold. To see this, note first that if $x\in\R^p$ and $\widehat{D}_{2,n}^{1/2}x\in \sC(S,\alpha)$, then
$$|\widehat{D}_{2,n}^{1/2}x_{S^c}|_1 = \sum_{j\in S^c}|\widehat{D}^{1/2}_{2,n}[j,j]|\,|x[j]|\leq \alpha\times|\widehat{D}_{2,n}^{1/2}x_{S}|_1 =\alpha \sum_{k\in S}|\widehat{D}^{1/2}_{2,n}[k,k]|\,|x[k]|.$$
{\sf Diag}$_-$ and {\sf Diag}$_+$ imply that 
\begin{equation}\label{eq:incone}|D^{1/2}_2\,x_{S^c}|_1\leq \alpha\,\sqrt{\frac{1+\eps}{1-\eps}}\,|D^{1/2}_2\,x_{S}|_1,\mbox{ ie. }D^{1/2}_2x\in\sC(S,\tilde{\alpha}).\end{equation}

This proves {\bf C1} in the Theorem, and also allows us to obtain:
\begin{eqnarray*}|\widehat{D}_{2,n}^{1/2}\,x|^2_1&\leq& (\alpha+1)^2\,|\widehat{D}_{2,n}^{1/2}x_{S}|^2_1 \\ &\leq& (1+\eps)\,(1+\alpha)^2\,|D_2^{1/2}x_{S}|^2_1 \\\mbox{(Cauchy Schwartz)}&\leq & (1+\eps)\,(1+\alpha)^2\,s\,|D_2^{1/2}x_{S}|^2_1\\ \mbox{(use \eqnref{incone})} &\leq &   \left(\frac{(1+\eps)\,(1+\alpha)^2}{\re(\bX,S,\tilde{\alpha})^2}\,s\right)\times [(D_2^{1/2}x)^T\bX\,(D_2^{1/2}x)] \\ \mbox{(defn. of $\bX$)} &=&  \left(\frac{(1+\eps)\,(1+\alpha)^2}{\re(\bX,S,\tilde{\alpha})^2}\,s\right)\,x^T\Sigma\,x\end{eqnarray*}
Combining the final bound with {\sf Lower} and using our assumption on $n$ we obtain
$$x^T\Sigmahat\,x\geq \left(1-\frac{\eps}{2}\right)\,x^T\Sigma\,x - \frac{392\,\hyp^2\,(1+2\ln(p/4\delta))}{\eps^2\,n}\,|\widehat{D}_{2,n}^{1/2}x|_1^2\geq (1-\eps)\,x^T\Sigma x.$$
This is {\bf C2}. Finally, note that ${\sf Diag}_-$ implies
\begin{equation}\label{eq:lowernorm}|\widehat{D}^{1/2}_{2,n}\,x|_1\leq \frac{|D_2^{1/2}\,x|_1}{\sqrt{1-\eps}}.\end{equation}
This means that if $\widehat{D}_{2,n}^{1/2}x\in \sC(S,\alpha)$, so that $D_2^{1/2}x\in \sC(S,\tilde{\alpha})$ (as shown above), we have
$$|\widehat{D}_{2,n}^{1/2}x_S|_2^2\leq \frac{|D_2^{1/2}x_S|_2^2}{(1-\eps)}\leq \frac{x^T\Sigma x}{(1-\eps)\,\re(\bX,S,\tilde{\alpha})^2}\leq \frac{x^T\Sigmahat\,x}{(1-\eps)^2\,\re(\bX,S,\tilde{\alpha})^{2}}.$$
Since this holds for any $\widehat{D}_{2,n}^{1/2}x$ as above we may use the substitution $y=\widehat{D}_{2,n}x$ to conclude:
$$\forall y\in \sC(S,\alpha)\,:\, (1-\eps)^2\,\re(\bX,S,\tilde{\alpha})^2\,|y_S|_ 2^2\leq y^T\bXhat\,y,$$
that is, {\bf C3} holds.

We have proved that {\bf C1}, {\bf C2} and {\bf C3} hold in the intersection {\sf Lower}$\cap${\sf Diag}$_+\cap${\sf Diag}$_-$. We now estimate the probability of this intersection by showing that each event has probability $\geq 1-\delta/3$. We already have this lower bound for {\sf Lower} from \lemref{intermediate}. For {\sf Diag}$_-$ we will use \lemref{concnonneg} in the appendix. Note that for each $1\leq j\leq p$
$$\widehat{D}_{2,n}[j,j] = \frac{1}{n}\sum_{i=1}^nX_i[j]^2$$
is a sum of $n$ i.i.d. non-negative random variables with mean $\Ex{X_1[j]^2}=D_2[j,j]$ and second moment
$$\Ex{X_1[j]^4}\leq \Ex{(e_i^TX_1)^4}\leq \hyp^2\,\Ex{(e_i^TX_1)^2}^2=\hyp^2\,D_2[j,j]^2.$$
Applying \lemref{concnonneg} for each $j$, with the choice $t = \eps^2n/2\hyp^2$, gives 
$$\Pr{{\sf Diag}_-}=\Pr{\forall 1\leq j\leq p\,:\, \widehat{D}_{2,n}[j,j]\geq (1-\eps)\,D_2[j,j]}\geq 1 - p\,e^{-\frac{\eps^2n}{2\hyp^2}}\leq \frac{\delta}{3}.$$
For {\sf Diag}$_+$ we will derive polynomial concentration bounds. We have already noted that, for each $1\leq j\leq p$, $\widehat{D}_{2,n}[j,j]$ is an average of $n$ i.i.d. random variables $X_i[j]^2$. Observe that $\Ex{X_1[j]^{2q}}^{1/q}\leq \hyp_*\,D_2[j,j]$ by assumption. Inequality \eqnref{BDG} above implies
$$\Ex{|\widehat{D}_{2,n}[j,j]-D_2[j,j]|^q}\leq n^{-q/2}\,(2q\,\hyp_*)^q\,D_2[j,j]^q.$$
Therefore,
$$\Pr{\widehat{D}_{2,n}[j,j]>(1+\eps)\,D_2[j,j]}\leq \frac{\Ex{|\widehat{D}_{2,n}[j,j]-D_2[j,j]|^q}}{\eps^q\,D_2[j,j]^q}\leq \left(\frac{(2q\,\hyp_*)^2}{\eps^2\,n}\right)^{\frac{q}{2}}$$
and a union bound over $j\in S$ implies:
$$\Pr{{\sf Diag}_+}=\Pr{\bigcup_{j\in S}\left\{\widehat{D}_{2,n}[j,j]>(1+\eps)\,D_2[j,j]\right\}}\leq \left(\frac{(2q\,\hyp_*)^2\,s^{2/q}}{\eps^2\,n}\right)^{\frac{q}{2}}\leq \frac{\delta}{3}$$
by our assumptions on the parameters. \end{proof}

\section{Final remarks}

\begin{itemize}
\item LASSO-type estimators like the one described here have been analyzed without restricted eigenvalue assumptions. Bartlett, Mendelson and Neeman \cite{BartlettMN2012} prove that the LASSO acts as a penalized least squares regressor satisfying a sharp oracle inequality. While we do not pursue this here, one could prove such a result under weak moment assumptions similar to those of \thmref{OLS}, but allowing for $n\gg \log p$. The recipe would be to start from  \lemref{intermediate} above and combine the ``self-normalization"~ideas in the proof of \thmref{re} with standard methods for obtaining sharp oracle inequalities \cite{Bartlett2008}. However, in this case the penalty of the LASSO would scale as $\bigoh{\sqrt{\ln p/n}\,|\widehat{D}_{4,n}^{1/4}|_1}$, where $\widehat{D}_{4,n}$ contains $n^{-1}\sum_iX_i[j]^4$ on the diagonal ($1\leq j\leq p$) and zeros elsewhere.
\item An interesting queston is whether one can use PAC Bayesian methods to impove upon the upper tail in random covariance matrix estimation. It seems that at least some of the constants in references \cite{MendelsonPaouris2014,SrivastavaVershynin2013} could be improved by this approach. The main idea would be to use self-normalized concentration inequalities to compensate for the lack of infinitely many moments.
\item Another intersting problem (in the context of \thmref{re}) is to investigate whether the PAC Bayesian method can improve on the known recovery guarantees for vectors with bounded entries \cite{RudelsonZhou2013,RudelsonV2008}. One may try to achieve this via a different choice of smoothig distribution and it seems likely that one of those choices will improve e.g. the best known bounds for sampling rows of an orthogonal matrix. This would also have some bearing on the performance of the LASSO. 
\end{itemize}

\appendix

\section{An improvement over the result of Rudelson and Zhou}\label{sec:appendix.rudelson}

In this appendix we discuss how our Transfer Principle (\lemref{transfer}) can be applied to obtain an improvement of a recent result by  Rudelson and Zhou \cite{RudelsonZhou2013}. The notation and definitions from \secref{resetup} are taken for granted. 

The goal of \cite{RudelsonZhou2013} was to show that if $\Sigmahat$ ``acts like"~$\Sigma$ over sparse vectors, it necessarily inherits restricted eigenvalue properties from $\Sigma$. This is important because dealing directly with the restricted eigenvalues might be complicated, whereas controlling $\Sigmahat$ over sparse vectors is typically much easier. Their precise result reads as follows\footnote{The reader should beware that our notation and our definition of restricted eigenvalues do {\em not} coincide with that of \cite{RudelsonZhou2013}. What follows is a ``translation"~to our language.}: 

\begin{theorem}[Theorem 3 in \cite{RudelsonZhou2013}]Let $1/5>\eps>0$, $\alpha>0$ and $s\in\{1,\dots,p\}$. Let $\Sigma\in\R^{p\times p}$ be a positive semidefinite matrix such that $\re(\Sigma,s,3\alpha)>0$. 
$$d\equiv \left\lceil s\,\left[ 1+ \frac{16\,(3\alpha)^2\,(3\alpha+1)}{\eps^2\,\re^2(\Sigma,s,3\alpha)}\,\left(\max_{1\leq j\leq p}\Sigma[j,j]\right)\right]\right\rceil.$$
Assume that $\Sigmahat\in\R^{p\times p}$ is also positive semidefinite and satisfies 
\begin{equation}\label{eq:conditionrudelson}\forall x\in\R^p\,:\,|x|_0\leq d\Rightarrow (1-\eps)^2\,x^T\Sigma\,x\leq x^T\Sigmahat\,x\leq (1+\eps)^2\,x^T\Sigma\,x.\end{equation}
Then
$$\re(\Sigmahat,s,\alpha)>(1-5\eps)\,\re(\Sigma,s,\alpha)>0.$$\end{theorem}

Our own theorem is this.
\begin{theorem}[Proven below]\label{thm:rudelson}Let $\Sigmahat,\Sigma\in\R^{p\times p}$, $S\subset \{1,\dots,p\}$ and $\eps\in (0,1/2)$, $\gamma>0$ be given. Define $s\equiv |S|$
$$d\equiv \left\lceil 1 + \frac{8\,(\gamma+\eps)\,s\,(1+\alpha)^2}{\eps\,\re(\Sigma,S,\alpha)^2}\,\left(\max_{1\leq j\leq p}\Sigma[j,j]\right)\,\re(\Sigma,S,\alpha)\right\rceil$$
and assume that 
\begin{eqnarray}\label{eq:controlsparse}\forall v\in \R^p\mbox{ with }|v|_0\leq d, &  v^T\Sigmahat v\geq (1-\eps)\,v^T\Sigma\,v\\ 
\label{eq:controlsparse2}\forall j\in \{1,\dots,p\},& \Sigmahat[j,j] \leq\,(1+\gamma)\,\Sigma[j,j].\end{eqnarray}Then 
$\forall x\in \sC(S,\alpha)\,:\, x^T\Sigmahat\,x\geq (1-3\eps/2)\,\,x^T\Sigma\,x.$
In particular, $$\re(\Sigmahat,S,\alpha)\geq (1-\eps)\,\re(\Sigma,S,\alpha).$$\end{theorem} 
The main conceptual difference between these two Theorems is that the condition \eqnref{conditionrudelson} implies \eqnref{controlsparse} and \eqnref{controlsparse2} with $\gamma\approx \eps$ (and a slightly different $\eps$). Moreover, we do not require bounds on $\re(\Sigma,s,3\alpha)$, and the numerical constants in our result are better. We also note that the full proof of \thmref{rudelson} (which includes \lemref{transfer} above) is quite simple and about two pages long.\\

\begin{proof}[of \thmref{rudelson}]~By our assumptions, we have 
$$\max_{1\leq j\leq p}\Sigmahat[j,j]-(1-\eps)\Sigma[j,j]\leq M\equiv (\gamma+\eps)\,\max_{1\leq \leq p}\Sigma[j,j].$$ 
We also have the condition $v^T\Sigmahat v\geq (1-\eps)\,v^T\Sigma v$ for all $d$-sparse $v$. We may apply \lemref{transfer} with $D=M\,I_{p\times p}$ to conclude:
\begin{equation}\label{eq:forallsimple}\forall x\in \R^p\,:\,  x^T\Sigmahat x\geq (1-\eps)\,x^T\Sigma x - \frac{M}{d-1}\,|x|_1^2.\end{equation}
We now restrict attention to $x\in \sC(S,\alpha)$, noting that
\begin{eqnarray*}|x|^2_1&\leq & (1+\alpha)^2\,|x_S|^2_1\\ \mbox{(Cauchy-Schwarz)}&\leq& (1+\alpha)^2\,|S|\,|x_S|^2_2\\ \mbox{(defn. of $\re$)}&\leq&(1+\alpha)^2\,s\times \re(\Sigma,S,\alpha)^{-2}\,x^T\Sigma\,x\\ &\leq & \frac{\eps\,(d-1)}{M}\,x^T\Sigma\,x\end{eqnarray*}
by the definitions of $M$ and $d$. Plugging this back into \eqnref{forallsimple} gives:
$$\forall x\in\sC(S,\alpha)\,:\,x^T\Sigmahat\,x\geq (1-3\eps/2)\,x^T\Sigma\,x.$$
By the definition of $\re(\Sigma,S,\alpha)$, we also have
$$\forall x\in\sC(S,\alpha)\,:\,x^T\Sigma\,x\geq \re(\Sigma,S,\alpha)^2\,|x_S|_2^2,$$
and
$$\re(\Sigmahat,S\alpha)\geq \sqrt{1-3\eps/2}\,\re(\Sigma,S,\alpha).$$
The proof finishes once we note that:
$$\forall 0<\xi<3/4\,:\,\sqrt{1 - \xi}\geq 1 - \frac{\xi}{2\sqrt{1-\xi}}\geq 1-\xi$$
and apply this to $\xi=3\eps/2$ (which is $\leq 3/4$ because $\eps\leq 1/2$).\end{proof}

\section{Technical estimates} 

\subsection{A simple lemma on powers of traces}\label{sec:powertrace}

\begin{lemma}\label{lem:powertrace}Suppose $A$ is a random positive-semidefinite, trace-class, linear operator over $\Hi$ such that $\Ex{\tr(A)}<+\infty$. Suppose that $q\geq 1$ and that there exists some $\hyp>0$ such that  $\Ex{(\ip{v}{Av})^q}\leq \hyp^q\,\Ex{\ip{v}{Av}}^q$ for all $v\in\Hi$. Then $\Ex{\tr(A)^q}\leq \hyp^p\,\Ex{\tr(A)}^q$.\end{lemma}
\begin{proof}Let $\{e_j\}_{j=1}^p$ be orthonormal vectors in $\Hi$. Then, by Minkovski and our assumption,
$$\Ex{(\sum_{j=1}^p\ip{e_j}{Ae_j})^q}^{1/q}\leq \sum_{j=1}^p\Ex{(\ip{e_j}{Ae_j})^q}^{1/q}\leq \hyp\,\sum_{j=1}^p\Ex{\ip{e_j}{Ae_j}}\leq \hyp\,\Ex{\tr(A)};$$
note that we also used  (implicitly) the fact that $\ip{e_j}{Ae_j}\geq 0$ for all $j$.\end{proof}

\subsection{Lower tail concentration for non-negative random variables}

\begin{lemma}\label{lem:concnonneg}Let $W_1,\dots,W_n\in[0,+\infty)$ be independent non-negative random variables with finite second moments. Then:
$$\Pr{\sum_{i=1}^n (W_i-\Ex{W_i})\leq -\sqrt{2t\sum_{i=1}^n\Ex{W_i^2}}}\leq e^{-t}.$$\end{lemma}
\begin{proof}Note that for any non-negative $x\in\R$ we have the inequality:
$$e^{-x}\leq 1 - x + \frac{x^2}{2}.$$
Apply this to $x=\xi W_i$ (with $\xi>0$) and integrate to deduce:
$$\Ex{e^{-\xi W_i}}\leq 1 - \xi\Ex{W_i} + \frac{\xi^2}{2}\,\Ex{W_i^2}\leq e^{-\xi\Ex{W_i} + \frac{\xi^2}{2}\,\Ex{W_i^2}}.$$
Now use the independent of the $W_i$'s to obtain:
\begin{equation}\label{eq:perbernstein}\forall \xi>0\,:\, \Ex{e^{-\xi \sum_{i=1}^n (W_i-\Ex{W_i})}} = \prod_{i=1}^n \Ex{e^{-\xi W_i + \xi \Ex{W_i}}}\leq e^{\frac{\xi^2}{2}\,\sum_{i=1}^n\Ex{W_i^2}}.\end{equation}
The usual Bernstein's trick finishes the proof. More specifically, we note that, for any $\lambda>0$
$$\Pr{\sum_{i=1}^n (W_i-\Ex{W_i})\leq -\lambda}\leq \inf_{\xi\geq 0}\Ex{e^{-\xi(\sum_{i=1}^n (W_i-\Ex{W_i}))}}e^{+\xi\lambda},$$
then bound the RHS via \eqnref{perbernstein} and optimize in $\xi$.\end{proof}


\subsection{Proof of \propref{Appmartingale}}\label{sec:Appmartingale}

\begin{proof}[of \propref{Appmartingale}]~The key step in the proof is to show that the sequence of random variables $U_0\equiv 1$,
$$U_i\equiv \exp\left(\xi N_i - \frac{\xi^2}{2}W_i\right)\,\,(1\leq i\leq n)$$
form a supermartingale with respect to the filtration $\sG_i$. Granted that, Optional Stopping gives$\Ex{U_T}\leq \Ex{U_0}=1$ for any stopping time $1\leq T\leq n$. Taking $T$ as the first $1\leq i\leq n$ such that $U_i>e^{t}$ (or $T=n$ if there is no such time), we may use Markov's inequality inequality to deduce:
$$\Pr{\exists 1\leq i\leq n\,:\, N_i>\frac{\xi}{2}\,W_i + \frac{t}{\xi}} = \Pr{U_T>e^{t}}\leq e^{-t}\Ex{U_T}\leq e^{-t},$$ 
which is the desired result.

To prove that $U_i$ is indeed a supermartingale, note that:
$$\frac{U_i}{U_{i-1}} = \exp\left(\xi D_i - \frac{\xi^2}{2}\,(\Ex{D_i^2\mid\sG_{i-1}} + D_i^2)\right).$$
where $D_i=N_i-N_{i-1}$. We need to show that $\Ex{U_i/U_{i-1}\mid\sG_{i-1}}\leq 1$. By considering the conditional distribution of $D_i$, we see that it suffices to show that 
$$(\star)\,\Ex{e^{\xi D - \frac{\xi^2}{2}\,(D^2+\Ex{D^2})}}\leq 1$$
for any square-integrable random variable $D$ with $\Ex{D}=0$. To prove $(\star)$, let $D'$ be an independent copy of $D$. We have:
$$\xi D - \frac{\xi^2}{2}\,(D^2+\Ex{D^2}) = \Ex{\xi (D-D') - \frac{\xi^2}{2}\,(D-D')^2\mid D}$$
and the conditional Jensen inequality implies:
$$\Ex{e^{\xi D - \frac{\xi^2}{2}\,(D^2+\Ex{D^2})}}\leq \Ex{\Ex{e^{\xi (D-D') - \frac{\xi^2}{2}(D-D')^2}\mid D}} = \Ex{e^{\xi (D-D') - \frac{\xi^2}{2}\,(D-D')^2}}.$$
Note that $D-D'$ is a symmetric random variable. Therefore:
$$\Ex{e^{\xi (D-D') - \frac{\xi^2}{2}\,(D-D')^2}\mid |D-D'|=a} = \frac{e^{\xi a}+e^{-\xi a}}{2}\,e^{-\frac{\xi^2a}{2}} = \cosh(\xi a)\,e^{-\frac{\xi^2a}{2}}\leq 1$$
since $\cosh(x)\leq e^{x^2/2}$ for all $x\in\R$ (just compare Taylor expansions). Integrating this last inequality gives $(\star)$ and finishes the proof.\end{proof}

\bibliography{biblio}
\bibliographystyle{plain}

\end{document}